\documentclass[12pt,a4paper]{article}

\usepackage{amsthm}
\usepackage{mathrsfs}
\usepackage{amsfonts}
\usepackage[utf8x]{inputenc}
\usepackage{amsfonts}
\usepackage{amssymb}
\usepackage{amsmath}
\usepackage{mathtools}

\usepackage{graphicx}
\usepackage{array,color}
\usepackage{arydshln}
\usepackage{bbding}
\usepackage{threeparttable}
\usepackage{float}
\usepackage{multirow}
\usepackage{bm}
\usepackage{subcaption}
\usepackage{bigdelim}
\usepackage{soulutf8}
\usepackage{enumitem}

\usepackage[ruled,vlined]{algorithm2e}
\usepackage{algorithmic}
\usepackage{hhline}
\usepackage{upgreek}
\usepackage{tikz}
\usetikzlibrary{decorations.pathreplacing}
\usetikzlibrary{shadows}
\usetikzlibrary{shapes}
\usepackage{soulutf8}

\usepackage{geometry}
\newtheorem{theorem}{Theorem}
\newtheorem{proposition}[theorem]{Proposition}
\newtheorem{rmk}{\textit{Remark}}[section]

\begin{document}

\title{New optimized Schwarz algorithms for one dimensional Schr\"odinger equation with general potential}
\author{
F. Xing \thanks{
Laboratoire de Math\'ematiques J.A. Dieudonn\'e, UMR 7351 CNRS, University Nice Sophia Antipolis,  Team COFFEE, INRIA Sophia Antipolis M\'editerran\'ee, Parc Valrose 06108 Nice Cedex 02, France,  BRGM Orl\'eans France}
}
\maketitle
\begin{abstract}
The aim of this paper is to develop new optimized Schwarz algorithms for the one dimensional Schr\"odinger equation with linear or nonlinear potential. 
After presenting the classical algorithm which is an iterative process, we propose a new algorithm for the Schr\"odinger equation with time-independent linear potential. 
Thanks to two main ingredients (constructing explicitly the interface problem and using a direct method on the interface problem), 
the new algorithm turns to be a direct process. 
Thus, it is free to choose the transmission condition.
Concerning the case of time-dependent linear potential or nonlinear potential, we propose to use a pre-processed linear operator as preconditioner which leads to a preconditioned algorithm. 
Numerically, the convergence is also independent of the transmission condition. In addition, both of these new algorithms implemented in parallel cluster are robust, scalable up to 256 sub domains (MPI process) and take much less computation time than the classical one, especially for the nonlinear case.
\end{abstract}
%

%
\section{Introduction}

This paper deals with the optimized Schwarz method without overlap for the one dimensional Schr{\"o}dinger equation defined on a bounded spatial domain $\Omega=(a_0,b_0)$, $a_0, b_0 \in \mathbb{R}$ and $t \in (0,T)$, which reads
\begin{equation}
  \label{Schequ}
  \left\{
    \begin{array}{ll}
      \mathscr{L}u := i\partial_t u + \partial_{xx} 
      u + \mathscr{V} u = 0, \ (t,x)\in (0,T)\times (a_0,b_0),\\
      u(0,x) = u_0(x), \ x \in (a_0,b_0),
    \end{array} 
  \right.
\end{equation}
where $\mathscr{L}$ is the Schr{\"o}dinger operator, the initial value $u_0 \in L^2(\mathbb{R})$ and the general real potential $\mathscr{V}$ consists of a linear and a nonlinear part
\begin{displaymath}
\mathscr{V} =V(t,x)+f(t,x,u).
\end{displaymath}
We complete the equation with homogeneous Neumann boundary conditions on the boundary $\partial_\mathbf{n} u (t,x) = 0, \ x=a_0,b_0$ where $\partial_\mathbf{n}$ denotes the normal directive, $\mathbf{n}$ being the outward unit vector on the boundary.

Among the various domain decomposition methods, we focus our attention on the optimized Schwarz method \cite{Gander2006osw}, which has been widely studied over the past years for many different equations, like the Poisson equation \cite{Lions1990}, the Helmholtz equation \cite{Boubendir2012,Gander2002helm} and the convection-diffusion equation \cite{NATAF1995}.

In order to perform the optimized Schwarz method, the equation \eqref{Schequ} is first semi-discretized in time on the entire domain $(0,T) \times (a_0,b_0)$. The time interval
$(0,T)$ is discretized uniformly with $N_T$ intervals
of length $\Delta t$. We denote by $u_{n}$ (resp. $V_n$) an approximation of the solution $u$ (resp. $V$) at time $t_n=n\Delta t$. The usual semi-discrete in time 
scheme developed by Durán and Sanz-Serna \cite{DUR2000} applied to \eqref{Schequ} reads as
\begin{displaymath}
\begin{split}
    i\frac{u_{n} - u_{n-1}}{\Delta t} & +  \partial_{xx} \frac{u_{n} + u_{n-1}}{2}  +  \frac{V_{n}  + V_{n-1}}{2} \frac{u_{n} + u_{n-1}}{2} 
    \\  & + f(t_{n+1/2},x,\frac{u_{n} + u_{n-1}}{2}) \frac{u_{n} + u_{n-1}}{2} = 0, \ 1 \leqslant n \leqslant N_T.
    \end{split}
\end{displaymath} 
By introducing new variables $v_{n} = (u_{n} + u_{n-1})/2$
with $v_{0} = u_{0}$ and $W_{n} =  (V_{n} + V_{n-1})/2$, we
get a stationary equation defined on $\Omega$ with unknown
$v_n$ 
\begin{equation}
  \label{EqSemi}
  \mathcal{L}_{\mathbf{x}} v_{n} := \frac{2i}{\Delta t} v_n + \partial_{xx}  v_n + W_n v_n + f(t_{n+1/2},x,v_n) v_n = \frac{2i}{\Delta t} u_{n-1},
\end{equation}
The original unknown $u_{n}$ is recovered by $u_{n} = 2 v_{n} - u_{n-1}$.

The equation \eqref{EqSemi} is stationary for any $1 \leqslant n \leqslant N_T$. We can therefore apply the optimized Schwarz method. Let us decompose the spatial domain $\Omega$ into $N$ sub domains $\Omega_j =(a_j,b_j),j=1,2,...,N$ without overlap as shown in Figure \ref{figsub3} for $N=3$. The classical Schwarz algorithm is an iterative process and we identify the iteration number thanks to label $k$. We denote by $v_{j,n}^k$ the solution on sub domain $\Omega_j$ at iteration $k=1,2,...$ of the Schwarz algorithm (resp $u_{j,n}^k$). Assuming that $u_{0,n-1}$ is known, the optimized Schwarz algorithm for \eqref{EqSemi} consists in applying the following sequence of iterations for $j=1,2,...,N$
\begin{equation}
  \label{algoglobal}
  \left \{
    \begin{array}{l}
      \displaystyle   \mathscr{L}_{\mathbf{x}}  v_{j,n}^{k} = \frac{2i}{\Delta t} u_{j,n-1}, \ (x,y) \in \Omega_j,  \\[2mm]
      \partial_{\mathbf{n}_j} v^{k}_{j,n} + S_j v^{k}_{j,n} = \partial_{\mathbf{n}_{j}} v^{k-1}_{j-1,n} + S_j v^{k-1}_{j-1,n} , \ x=a_j,\\[2mm]
      \partial_{\mathbf{n}_j} v^{k}_{j,n} + S_j v^{k}_{j,n} = \partial_{\mathbf{n}_{j}} v^{k-1}_{j+1,n} + S_j v^{k-1}_{j+1,n}, \ x=b_j,
    \end{array} \right.
\end{equation}
with a special treatment for the two extreme sub domains
$\Omega_1$ and $\Omega_N$ since
the boundary conditions are imposed on $\{x=a_1\}$ and $\{x=b_N\}$ 
\begin{displaymath}
  \partial_{\mathbf{n}_1} v^{k}_{1,n}  = 0, x=a_1, \quad 
  \partial_{\mathbf{n}_N} v^{k}_{N,n}  = 0, \ x=b_N,
\end{displaymath} 
where $S_j$ is the transmission operator.

\definecolor{col1}{rgb}{0.8,1,0.8}
\definecolor{col2}{rgb}{1,0.8,0.8}
\definecolor{col3}{rgb}{0.8,0.8,1}
\begin{figure}[!htbp] 
  \centering
  \begin{tikzpicture}
    \draw[>=stealth,->, line width=2.5pt] (-0.5,0) -- (8,0);
    \draw (8,0) node[right] {$x$};

    \draw[line width=1.5pt] (0.2,-0.15) -- (0.2,0.15); 
    \draw[line width=1.5pt] (2.0,-0.15) -- (2.0,0.15);
    \draw[line width=1.5pt] (4.6,-0.15) -- (4.6,0.15); 
    \draw[line width=1.5pt] (7.5,-0.15) -- (7.5,0.15);     
    
    \draw (0.2,-0.2) node[below] {$a_0=a_1$};
    \draw (2,-0.1) node[below] {$b_1=a_2$};
    \draw (4.6,-0.1) node[below] {$b_2=a_3$};
    \draw (7.5,-0.1) node[below] {$b_2=b_0$};


\draw (1.1,0.52) node[below] {$\Omega_1$};
\draw (3.35,0.52) node[below] {$\Omega_2$};
\draw (6.12,0.52) node[below] {$\Omega_3$};
  \end{tikzpicture}
  \caption{Domain decomposition without overlap, $N=3$.}
  \label{figsub3}
\end{figure}
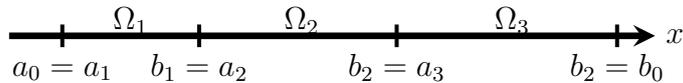

The transmission condition is an important issue for this method. In recent years, many works contribute to look for good transmission conditions to obtain faster convergence in the context of the Schwarz waveform relaxation method and the optimized Schwarz method for the Schr\"odinger equation. The widely used Robin transmission condition 
\begin{equation}
\label{condrobin}
 \mathrm{Robin:} \ S_j v = -ip \cdot v, \ p \in \mathbb{R}^+ \text{ is parameter},
\end{equation}
the optimal transmission condition for constant potential etc. are considered in \cite{Halpern2010_sch}. 
The authors of \cite{Antoine2014, XF20151d, XF20152dlin} introduce the idea using some absorbing boundary operators \cite{Antoine2009a,Antoine2011} as the transmission operator. For example,
\begin{equation}
\label{TCS2P}
\begin{split}
     S_{\mathrm{pade}}^m: \ S_j v  & = -i \sqrt{\frac{2i}{\Delta t} + V + f(v) } v ,  \\
  & \approx
    \Big( -i\sum_{s=0}^m a_s^m + i\sum_{s=1}^m a_s^m d_s^m (\frac{2i}{\Delta t} + V + f(v) + d_s^m)^{-1} \Big) v,
    \end{split}
    \end{equation}
where $a_s^m = e^{i\theta/2}/\big(m\cos^2(\frac{(2s-1)\pi}{4m}) \big)$, $d_s^m = e^{i\theta} \tan^2 (\frac{(2s-1)\pi}{4m})$, $s=0,1,...,m$ are constant coefficients associate with the Pad\'e approximation of order $m$.
These operators are constructed by using some pseudo differential techniques and could be seen as different approximate accesses to the optimal operator. However, numerically  there are always some parameters in these transmission conditions that need to be optimized for each potential $\mathscr{V}$, each size of time step $\Delta t$, each mesh and each number of sub domains $N$. This leads to huge extra tests before launching the simulation.

The parallel scalability is one of the key objectives of this method.
The authors of \cite{XF20151d} introduced some new efficient and robust algorithms in the framework of the Schwarz waveform relaxation algorithm for the one dimensional Schr\"odinger equation. Then, the new algorithms are extended to the two dimensional linear case in the context of the optimized Schwarz algorithm \cite{XF20152dlin}. These new algorithms are much more scalable and efficient than the classical one. However, choosing the transmission condition and the parameters in the transmission condition remain very inconvenient.

The objective of this article is to develop scalable optimized Schwarz algorithms on parallel computers. We propose some new algorithms based on the studies of the classical algorithm. These new algorithms have following main features theoretically or numerically: the convergence is independent of the transmission condition, the size of time step $\Delta t$, the mesh and the number of sub domains. In addition, the computation time is scalable and much less compared with the classical one.

This paper is organized as follows. In section \ref{sec_cls}, we study the classical algorithm and construct the interface problem. In Section \ref{sec_new} and \ref{sec_pd}, we present the new algorithm for the Schr\"odinger equation with time-independent linear potential and the preconditioned algorithm for  the Schr\"odinger equation with general potential. Some numerical results are shown in Section \ref{sec_num}. Finally, we draw a conclusion.

\begin{rmk}
To simplify the presentation, we only consider the Robin transmission condition \eqref{condrobin} in the following paper. Applying other more complex transmission conditions in \cite{Halpern2010_sch, Antoine2014, XF20151d} such as \eqref{TCS2P} is direct, however they are not useful in our new algorithms.
\end{rmk}

\section{Classical optimized Schwarz algorithm}
\label{sec_cls}

Let us introduce the fluxes $l_{j,n}^k$ and $r_{j,n}^k$, $j=1,2,...,N$ defined as
\begin{displaymath}
    l_{j,n}^k = \partial_{\mathbf{n}_j} v^{k}_{j,n}(a_j) + S_j v^{k}_{j,n} (a_j), \quad
    r_{j,n}^k = \partial_{\mathbf{n}_j} v^{k}_{j,n}(b_j) + S_j v^{k}_{j,n} (b_j), 
\end{displaymath}
with a special definition for the two extreme sub domains: $l_{1,n}^k = r_{N,n}^k=0$. Thus, the algorithm \eqref{algoglobal} can be splitted into local problems on sub domain $\Omega_j,j=1,2,...,N$
\begin{equation}
  \label{algolocal}
  \left \{
    \begin{array}{l}
      \displaystyle   \mathscr{L}_{\mathbf{x}}  v_{j,n}^{k} = \frac{2i}{\Delta t} u_{j,n-1}, \\[2mm]
      \partial_{\mathbf{n}_j} v^{k}_{j,n} + S_j v^{k}_{j,n} = l_{j,n}^k, \ x=a_j,\\[2mm]
      \partial_{\mathbf{n}_j} v^{k}_{j,n} + S_j v^{k}_{j,n} = r_{j,n}^k, \ x=b_j,
    \end{array} \right.
\end{equation}
and flux problems 
\begin{equation}
  \label{fluxlr}
  \left \{
    \begin{array}{l}
      l_{j,n}^{k+1} = -r_{j-1,n}^{k} + 2 S_j v_{j-1,n}^k(b_{j-1}), \ j=2,3,...,N,\\[2mm]
      r_{j,n}^{k+1} = -l_{j+1,n}^{k} + 2 S_j v_{j+1,n}^k(a_{j+1}),\ j=1,2,...,N-1.
    \end{array} \right. 
\end{equation}    

The spatial approximation is realized by the standard $\mathbb{P}_1$ finite element method. If $f=0$, then the system \eqref{algolocal} is linear. Let us denote by $\mathbf{v}_{j,n}^k$ (resp. $\mathbf{u}_{j,n}^k$) the nodal interpolation vector of $v_{j,n}^k$ (resp. $u_{j,n}^k$) with $N_j$ nodes, $\mathbb{M}_j$ the mass matrix, $\mathbb{S}_j$ the stiffness matrix,  $\mathbb{M}_{j,W_n}$ the generalized mass matrix with respect to $\int_{a_j}^{b_j} W_n v \phi dx$. Then, the matrix formulation of the $N$ local problems is given by
\begin{gather}
  \Big( \mathbb{A}_{j,n} + ip \cdot \mathbb{M}^{\Gamma_j} \Big)
  \mathbf{v}_{j,n}^k =  \frac{2i}{\Delta t} \mathbb{M}_j
  \mathbf{u}_{j,n-1}^k - \mathbb{M}^{\Gamma_j}  Q_j^{\top} 
  \begin{pmatrix}
    l_{j,n}^k\\
    r_{j,n}^k
  \end{pmatrix},  \label{distrobin} 
\end{gather} 
where $\mathbb{A}_{j,n} = \frac{2i}{\Delta t}\mathbb{M}_j - \mathbb{S}_j + \mathbb{M}_{j,W_{n}}$ and "$\cdot^{\top}$" is the standard notation of the transpose of a matrix or a vector. The boundary matrix $\mathbb{M}^{\Gamma_j}$  and the restriction matrix $Q_j$ are defined as
\begin{displaymath}
\mathbb{M}^{\Gamma_j} =
\begin{pmatrix}
1 & 0 & \cdots & 0 \\
0 & 0 & \cdots & 0 \\
\vdots & \vdots & & \vdots \\
0 & 0 & \cdots & 1
\end{pmatrix}
\in \mathbb{C}^{N_j \times N_j}, \quad
  Q_{j} = \begin{pmatrix}
    1 & 0 & 0 & \cdots & 0 & 0 \\
    0 & 0 & 0 & \cdots & 0 & 1
  \end{pmatrix}
  \in \mathbb{C}^{2 \times N_j}.
\end{displaymath}

If $f\neq 0$, then the equation \eqref{algolocal} is nonlinear. Assuming that $u_{j,n-1}$ is known, the computation of $v^{k}_{j,n}$ is accomplished by a fixed point procedure. We take $\zeta^{0}_j=u_{j,n-1}$ and compute the solution $v^{k}_{j,n}$ as the limit of the iterative procedure with respect to the label $q$, $q=1,2,...$
\begin{equation}
  \label{localpbrobin}
  \left \{
    \begin{array}{l}
      \displaystyle  \Big( \frac{2i}{\Delta t} + \partial_{xx} + W_n  \Big) \zeta^q_j = \frac{2i}{\Delta t} u_{j,n-1} - f(t_{n+1/2},x,\zeta^{q-1}_j) \zeta^{q-1}_{j} , \\[2mm]
      \partial_{\mathbf{n}} \zeta^q_j - ip \cdot \zeta^q_j = l_{j,n}^{k}, \ x=a_j,\\[2mm]
      \partial_{\mathbf{n}} \zeta^q_j -ip \cdot \zeta^q_j = r_{j,n}^{k}, \ x=b_j.
    \end{array} \right.
\end{equation}
Let us denote by $\bm{\zeta}^q_j$ the nodal interpolation of $\zeta^q_j$ and $\mathbf{b}_{f(t_{n+1/2},x,\zeta^{q}_j)}$ the vector with respect to $\int_{a_j}^{b_j} f(t_{n+1/2},x,\zeta^{q}_j) \zeta^{q}_j \phi dx$. Thus, the matrix formulation of \eqref{localpbrobin} is 
\begin{equation}
  \label{AvMNL}
  \quad \Big( \mathbb{A}_{j,n} + ip \cdot \mathbb{M}^{\Gamma_j} \Big)
  \bm{\zeta}^{q}_j =  \frac{2i}{\Delta t} \mathbb{M}_j
  \mathbf{u}_{j,n-1}^{k} - \mathbf{b}_{f(t_{n+1/2},x,\zeta^{q-1}_j)} - Q_j^{\top} 
  \begin{pmatrix}
    l_{n}^k\\
    r_{n}^k
  \end{pmatrix}.
\end{equation}

In addition, the discrete formulation of the flux problems \eqref{fluxlr} reads
\begin{equation}
  \label{TransmissionCond_disc}
  \left \{ 
    \begin{array}{ll}
      l_{j,n}^{k+1} = -r_{j-1,n}^{k} - 2 ip \cdot Q_{j-1,r} \mathbf{v}_{j-1,n}^{k}, \ j=1,2,...,N-1,\\[2mm]
      r_{j,n}^{k+1} = -l_{j+1,n}^{k} - 2 ip \cdot Q_{j+1,l} \mathbf{v}_{j+1,n}^{k},\ j=2,3,...,N.
    \end{array} \right.
\end{equation} 
where $Q_{j,l} = (1, 0, \cdots, 0, 0 ) \in \mathbb{C}^{N_j}$, $Q_{j,r} = (0, 0, \cdots, 0, 1) \in \mathbb{C}^{N_j}$.

The classical algorithm is initialized by an initial
guess of $l_{j,n}^0$ and $r_{j,n}^0$, $j=1,2,...,N$. The
boundary conditions for any sub domain $\Omega_j$ at
iteration $k+1$ involve the knowledge of the values of
the functions on adjacent sub domains $\Omega_{j-1}$ and
$\Omega_{j+1}$ at prior iteration $k$. Thanks to the
initial guess, we can \emph{solve} the Schr\"odinger equation on
each sub domain, allowing to build the new boundary
conditions for the next step, \emph{communicating} them to other
sub domains. This procedure is summarized 
in \eqref{AlgoGraph} for $N=3$ sub domains at iteration
$k+1$.
\begin{equation}
  \label{AlgoGraph}
  \begin{pmatrix}
    r_{1,n}^k \\
    l_{2,n}^k \\
    r_{2,n}^k \\
    r_{3,n}^k
  \end{pmatrix}
  \underrightarrow{\hspace{0.1cm} \text{\footnotesize{\emph{Solve}}} \hspace{0.1cm}}
  \begin{pmatrix}
    \mathbf{v}_{1,n}^k \\
    \mathbf{v}_{2,n}^k \\
    \mathbf{v}_{3,n}^k \\
  \end{pmatrix}
  \underrightarrow{ }
  \begin{pmatrix}
    - r_{1,n}^k - 2ip \cdot {Q}_{1,r} \mathbf{v}_{1,n}^k \\
    - l_{2,n}^k  - 2ip \cdot{Q}_{2,l} \mathbf{v}_{j,n}^k \\
    - r_{2,n}^k - 2ip \cdot {Q}_{2,r} \mathbf{v}_{j,n}^k \\
    - l_{3,n}^k  - 2ip \cdot {Q}_{3,l} \mathbf{v}_{N,n}^k
  \end{pmatrix}
  \underrightarrow{ \hspace{0.1cm} \text{\footnotesize{\emph{Comm.}}} \hspace{0.1cm} }
  \begin{pmatrix}
    r_{1,n}^{k+1} \\
    l_{2,n}^{k+1} \\
    r_{2,n}^{k+1} \\
    l_{3,n}^{k+1} \\
  \end{pmatrix}.
\end{equation}    
Let us define the discrete interface vector by 
\begin{displaymath}
  \mathbf{g}_n^{k}=(r_{1,n}^{k},\cdots,l_{j,n}^{k},r_{j,n}^{k},\cdots,l_{N,n}^{k})^{\top}.
\end{displaymath}
Thanks to this definition, we give a new interpretation to the algorithm which can be written as
\begin{equation}
  \label{algo_d}
  \mathbf{g}_n^{k+1} = \mathcal{R}_{h,n} \mathbf{g}_n^k = I - (I - \mathcal{R}_{h,n}) \mathbf{g}_n^k.
\end{equation} 
where $I$ is identity operator and $\mathcal{R}_{h,n}$ is an operator. The solution to this iteration process is given as the
solution to the discrete interface problem
\begin{equation*}
  \label{interfacepb_d}
  ( I - \mathcal{R}_{h,n}) \mathbf{g}_n = 0.
\end{equation*}

\section{New algorithm for time-independent linear potential}
\label{sec_new}

The first aim of this section is to show that if the potential is linear $f=0$, then the discrete interface problem can be written as
\begin{equation}
  \label{InterfacePbdisctILgd}
  (I - \mathcal{L}_{h,n}) \mathbf{g}_n = \mathbf{d}_n,
\end{equation}
where $\mathbf{d}_n$ is a vector and $\mathcal{L}_{h,n} \in \mathbb{C}^{(2N-2)\times(2N-2)}$ is a matrix (the notation ``MPI $j$''
  above the columns of the matrix will be used after in this section)
\begin{equation}
\setlength\arraycolsep{3pt}
  \label{matL}
  \mathcal{L}_{h,n} = 
  \begin{pmatrix}
    \multicolumn{1}{l}{\overbrace{\hspace{2.0em}}^{\mathrm{MPI}\ 0}} & 
    \multicolumn{2}{l}{\overbrace{\hspace{4.2em}}^{\mathrm{MPI}\ 1}} &
    \multicolumn{2}{l}{\overbrace{\hspace{4.2em}}^{\mathrm{MPI}\ 2}} &
    & 
    \multicolumn{2}{l}{\overbrace{\hspace{7.0em}}^{\mathrm{MPI}\ N-2}} &
    \multicolumn{2}{l}{\overbrace{\hspace{2.0em}}^{\mathrm{MPI}\ N-1}}
    \\
    & x^{2,1}_n & x^{2,2}_n & & & \\
    x^{1,4}_n \\
    & & & x^{3,1}_n & x^{3,2}_n \\
    & x^{2,3}_n & x^{2,4}_n \\
    & & & & & \cdots \\
    & & & x^{3,3}_n & x^{3,4}_n \\
    & & & & & & x^{N-1,1}_n & x^{N-1,2}_n\\
    & & & & &\cdots \\
    & & & & & & & & x^{N,1}_n\\
    & & & & & & x^{N-1,3}_n & x^{N-1,4}_n
  \end{pmatrix}.
\end{equation}                                                     
Especially, if the linear potential is time-independent $V=V(x)$, then
\begin{equation}
\label{Lhequ}
\mathcal{L}_{h,1}=\mathcal{L}_{h,2}=...=\mathcal{L}_{h,N_T}.
\end{equation}

Secondly, based on these observations, we propose a new Schwarz algorithm for the Schr\"odinger equation with time-independent linear potential.

\begin{proposition}
  \label{proprobin}
  Assuming that $f=0$, then the operator $\mathcal{R}_{h,n}$ is linear
  \begin{displaymath}
    \mathcal{R}_{h,n}  \mathbf{g}_n^{k} = \mathcal{L}_{h,n} \mathbf{g}_n^{k} + \mathbf{d}_{n} 
  \end{displaymath}  
  where $\mathcal{L}_{h,n}$ is a matrix as shown by \eqref{matL}. In addition, if the linear potential is time-independent $V=V(x)$, then the interface matrix $\mathcal{L}_{h,n}$ satisfies \eqref{Lhequ}.
\end{proposition}

\begin{proof}
  Firstly, by some straight forward calculations using \eqref{distrobin} and \eqref{TransmissionCond_disc}, we can verify that
  \begin{equation}
    \label{matXRobin}
    \begin{aligned}
      & x^{j,1}_n =  -I - 2ip \cdot {Q}_{j,l} (\mathbb{A}_{j,n}
      +ip \cdot \mathbb{M}^{\Gamma_j})^{-1}
      \mathbb{M}^{\Gamma_{j}} {Q}_{j,l}^{\top}, \\
      & x^{j,2}_n =  - 2ip \cdot {Q}_{j,l} (\mathbb{A}_{j,n}
      +ip \cdot \mathbb{M}^{\Gamma_j})^{-1}
      \mathbb{M}^{\Gamma_{j}} {Q}_{j,r}^{\top}, \\
      & x^{j,3}_n =  - 2ip \cdot {Q}_{j,r} (\mathbb{A}_{j,n}
      +ip \cdot \mathbb{M}^{\Gamma_j})^{-1}
      \mathbb{M}^{\Gamma_{j}} {Q}_{j,l}^{\top}, \\
      & x^{j,4}_n =  -I - 2ip \cdot {Q}_{j,r} (\mathbb{A}_{j,n}
      +ip \cdot \mathbb{M}^{\Gamma_j})^{-1}
      \mathbb{M}^{\Gamma_{j}} {Q}_{j,r}^{\top},
    \end{aligned}
  \end{equation}  
  and $\mathbf{d}_{n} = (d_{n,1,r},..., d_{n,j,l}, d_{n,j,r}, ..., d_{n,N,r})^{\top}$ with
  \begin{equation}
    \label{dRobin}
    \begin{split}
      d_{n,j,l} & = 2ip \cdot \frac{2i}{\Delta t} \cdot {Q}_{j-1,r} (\mathbb{A}_{j-1,n} +ip \cdot \mathbb{M}^{\Gamma_{j-1}})^{-1} \mathbb{M}_{j-1} \mathbf{u}_{j-1,n}, \ j=2,3,...,N, \\
      d_{n,j,r} & = 2ip \cdot \frac{2i}{\Delta t} \cdot {Q}_{j+1,l} (\mathbb{A}_{j+1,n} +ip \cdot \mathbb{M}^{\Gamma_{j+1}})^{-1} \mathbb{M}_{j+1} \mathbf{u}_{j+1,n}, \ j=1,2,...,N-1.
    \end{split}
  \end{equation}
  Secondly, since $V=V(x)$, then for $j=1,2,...,N$, we have
  \begin{gather*}
    \label{Aeq}
    \mathbb{M}_{j,W_1}=\mathbb{M}_{j,W_2}=...=\mathbb{M}_{j,W_{N_T}}, \\  \Rightarrow \mathbb{A}_{j,1}=\mathbb{A}_{j,2}=...=\mathbb{A}_{j,N_T} = \frac{2i}{\Delta t} \mathbb{M}_{j}-\mathbb{S}_{j} + \mathbb{M}_{j,W_n}.
  \end{gather*}
  Thus, the elements of $\mathcal{L}_{h,n}$ satisfy
  \begin{displaymath}
  x^{j,s}_1 = x^{j,s}_2 =... =x^{j,s}_{N_T},
  \end{displaymath}
  for $j=1,2,...,N$, $s=1,2,3,4$.
\end{proof}

Thanks to the analysis yielded in the previous proposition, we can build explicitly $\mathcal{L}_{n,h}$ with not much computation.  
According to \eqref{matXRobin}, to know the elements $x^{j,1}$ and $x^{j,3}$ (resp. $x^{j,2}$ and $x^{j,4}$), it is necessary to compute one time the application of $ (\mathbb{A}_{j,n} + ip \cdot \mathbb{M}^{\Gamma_j})^{-1}$ to vector $\mathbb{M}^{\Gamma_{j}} Q_{j,l}^{\top}$ (resp. $\mathbb{M}^{\Gamma_{j}} Q_{j,r}^{\top}$).
%
%
In total, to know the four elements $x^{j,1}_n$, $x^{j,2}_n$, $x^{j,3}_n$ and $x^{j,4}_n$, it is enough to compute two times the application of $ (\mathbb{A}_{j,n} + ip \cdot \mathbb{M}^{\Gamma_j})^{-1}$ to vector. In other words, this amounts to solve the Schr\"odinger equation on each sub domain two times to build the matrix $\mathcal{L}_{h,n}$. In addition, let us note that $\mathcal{L}_{h,1}=\mathcal{L}_{h,2}=...=\mathcal{L}_{h,N_T}$ if $V=V(x)$.

The construction of the vector $\mathbf{d}_n$ is similar. According to
\eqref{dRobin}, one needs to apply $(\mathbb{A}_{j,n} +ip \cdot \mathbb{M}^{\Gamma_{j}})^{-1}$ to vector $\mathbb{M}_{j} \mathbf{u}_{j,n}$ for $j=1,2,...,N$. In
other words, we solve the Schr\"odinger equation on each sub domain
one time. The vector is stored in a distributed manner using the PETSc library \cite{petsc-user-ref}
\begin{equation*}
  \mathbf{d}_n = \big( 
\underbrace{d_{n,1,r}}_{\text{MPI 0}}, 
\underbrace{d_{n,2,l}, d_{n,2,r}}_{\text{MPI 1}}, \cdots,\underbrace{d_{n,j,l}, d_{n,j,r}}_{\text{MPI j-1}}, \cdots, \underbrace{d_{n,N,l}}_{\text{MPI N-1}} \big)^{\top}.
\end{equation*}

Let us now present the new algorithm for the Schr\"odinger equation with time-independent linear potential.
\begin{center}
  \begin{minipage}{0.9\textwidth}
    \begin{algorithm}[H]
      \caption{New algorithm, time-independent linear potential}
      \label{Algo_new}
      {\footnotesize 1:} Build the interface matrix $\mathcal{L}_{h,n}$ explicitly.\\
      {\footnotesize 2:} The initial datum is $u_{0}$. \\
      \For{$n=1,2,...,N_T$}{
        {\footnotesize 2.1:} Build the vector $\mathbf{d}_{n}$ on time step $n$,\\
        {\footnotesize 2.2:} Solve the linear system
        \begin{equation}
          \label{newalgoILd}
          (I-\mathcal{L}_{h,n})\mathbf{g}_{n}=\mathbf{d}_{n}.
        \end{equation}        \\
        {\footnotesize 2.3:} Solve the Schr\"{o}dinger equation on time step $n$ on each sub domain using the flux from step 2.2 to compute $u_{n}$. 
      }
    \end{algorithm}
  \end{minipage}
\end{center}

Concerning the computational phase 2.2 in the new algorithm and the storage of the matrix $\mathcal{L}_{h,n}$, there are some choices.
\begin{itemize}[leftmargin=0pt]
\item[] \textbf{Choice 1.} The transpose of $\mathcal{L}_{h,n}$ is stored in a distributed manner using the library PETSc. As shown by \eqref{matXRobin}, the first column of $\mathcal{L}_{h,n}$ lies in MPI process 0. The second and third columns are in MPI process 1, and so on for other processes. The linear system \eqref{newalgoILd} is solved by any iterative methods, such as the fixed point method and the Krylov methods (GMRES, BiCGStab) \cite{Saad2003}. Note that if the fixed point method is used, then formally we recover the classical algorithm \eqref{algo_d}.
  %
  %
\item[] \textbf{Choice 2.} The transpose of $\mathcal{L}_{h,n}$ is stored in a distributed manner using the library PETSc. Unlike the first choice, we solve the linear system \eqref{newalgoILd} by a parallel LU direct method (MPI). One obvious benefit is that the algorithm has no convergence problem. Thus, it is possible to use any $p \in \mathbb{R}^+$ in the Robin transmission condition without looking for the optimal one which makes the algorithm converge with less iterations for each size of time step, each mesh and each number of sub domains. 
  
The size of $\mathcal{L}_{h,n}$ is $2N-2$, which is small for a LU direct method even when $N$ is large, such as $N=256$. 
Solving the linear system needs few arithmetic operations.
However, two points are not negligible. It is well known that the LU direct method is inherently sequential. In addition, one MPI process contains only two rows of $\mathcal{L}_{h,n}^{\top}$ (four  elements). 
Thus, the parallel LU direct method could suffer considerable communication costs.
\item[] \textbf{Choice 3.} The matrix $\mathcal{L}_{h,n}$ is stored locally in one MPI process after construction. For $n=1,2,...,N_T$, the distributed vector $\mathbf{d}_n$ is firstly gathered to one MPI process, then the linear system \eqref{newalgoILd} is solved by a sequential LU direct method. Finally, the solution $\mathbf{g}_n$ is broadcased to all MPI processes. This choice is a variation of the second one. It can avoid communications from the parallel implementation of the LU method.
\end{itemize}
%


\begin{rmk}
In the new algorithm with the choice 2 or 3, it is enough to do one time the LU factorization of $I-\mathcal{L}_{h,n}$.
\end{rmk}

There are two main novelties in the context of the optimized Schwarz method for the one dimensional Schr\"odinger equation. We construct explicitly the matrix $\mathcal{L}_{h,n}$, while in the classical algorithm, $\mathcal{L}_{h,n}$ remains an abstract operator. It is not usual to build explicitly such huge operator, but as we have seen, its computation is not costly. 
We use the direct method on the interface problem, thus the algorithm is independent of the transmission condition since the convergence properties of $I-\mathcal{L}_{h,h}$ has no influence on direct linear solver.


\section{Preconditioned algorithm for general potential}
\label{sec_pd}

In the previous sections, we have interpreted the classical algorithm for the Schr{\"o}dinger equation as \eqref{algo_d}. However, if $f \neq 0$, then the operator $\mathcal{R}_{h,n}$ is nonlinear. If $f=0$ and $V=V(t,x)$, it is necessary to construct the interface matrix $\mathcal{L}_{h,n}$ for each time step $n$. Thus, the new algorithm is not suitable here. Instead, to reduce the number of iterations required for convergence, inspired from \cite{XF20151d}, we add a preconditioner $P^{-1}$ ($P$ is a non singular matrix) in \eqref{algo_d} or \eqref{InterfacePbdisctILgd} which leads to the preconditioned algorithm:
\begin{itemize}[leftmargin=*]
\item for $\mathscr{V}=V(t,x)$,
  \begin{align}
    & \mathbf{g}_n^{k+1} = I - P^{-1} (I - \mathcal{R}_{h,n}), \label{chp2_algopd_L} \\
    & P^{-1} (I - \mathcal{L}_{h,n})\mathbf{g}_n = P^{-1} \mathbf{d}_n,   \label{chp2_algopd_Lproof} 
  \end{align}
\item for $\mathscr{V}=V(t,x)+f(t,x,u)$,
  \begin{equation}
    \label{algopd_NL}
    \mathbf{g}_n^{k+1} = I - P^{-1} (I - \mathcal{R}_{h,n}) \mathbf{g}_n^k.
  \end{equation} 
\end{itemize}	

We now turn to explain which preconditioner is used.
The interface problem for the free Schr\"odinger equation ($V=0$, $f=0$) is 
\begin{displaymath}
  (I - \mathcal{L}_{0})\mathbf{g}_n = \mathbf{d}_n,
\end{displaymath}
where the symbol $\mathcal{L}_0$ is used to highlight here the potential is zero. The transmission condition is the same as for \eqref{Schequ}. We propose for time-dependent linear or nonlinear potential the preconditioner as
\begin{displaymath}
  P = I - \mathcal{L}_0.
\end{displaymath}
We have three reasons to believe that this is a good choice.
%
\begin{enumerate}[leftmargin=*]
\item
  {
    Intuitively, the Schr{\"o}dinger operator without potential is a roughly approximating of the Schr{\"o}dinger operator with potential:
    \begin{displaymath}
      i\partial_tu + \partial_{xx}u \approx i\partial_t u + \partial_{xx}u  + Vu + f(t,x,u)u.
    \end{displaymath} 
    In other words, $Vu+f(t,x,u)u$ is a perturbation of the free Schr\"odinger operator.

    If $\mathscr{V}=0$, then $I - \mathcal{L}_0 = I - \mathcal{L}_{h,n} = I - (\mathcal{R}_{h,n}-\mathcal{R}_{h,n} \cdot \mathbf{0})$ where $\mathbf{0}$ is the zero vector. Thus, the matrix $\mathcal{L}_{0}$ could be a good approximation of the matrix $\mathcal{L}_{h,n}$ and of the nonlinear operator $\mathcal{R}_{h,n}-\mathcal{R}_{h,n} \cdot \mathbf{0}$:
    \begin{equation}
      \label{PL0}
      P = I - \mathcal{L}_0  \approx I - \mathcal{L}_{h,n},\quad
      P = I - \mathcal{L}_0  \approx I - (\mathcal{R}_{h,n}-\mathcal{R}_{h,n} \cdot \mathbf{0}),
    \end{equation}
  }
  \item
  The matrix $\mathcal{L}_0$ can be constructed easily as explained in the previous section.
  \item
  The implementation of the preconditioner and the storage of $P$ are same as that of \eqref{newalgoILd}, which are discussed in the previous section by the three choices. In effect, for any vector $y$, the vector $z:=P^{-1} y$ is computed by solving the linear system
\begin{equation}
  \label{Pxg}
  P z = (I - \mathcal{L}_0) z= y,
\end{equation}   
as the inverse of a matrix numerically is too expensive. 
Note that using the choice 2 or 3 ensures that the application of the preconditioner is robust.
\end{enumerate}
%
%

%

\section{Numerical results}
\label{sec_num}

The physical domain $(a_0,b_0)=(-16,16)$ is decomposed into $N$ equal sub domains without overlap. The final time is fixed to be $T=1.0$. The initial data is
\begin{displaymath}
  u_0(x) = e^{-(x+1)^2+i(x+1)}.
\end{displaymath}
Let us consider three different potentials in the following three sub sections
\begin{itemize}
\item time-independent linear potential: $\mathscr{V}=-x^2$, 
\item time-dependent linear potential: $\mathscr{V}=5tx$,
\item nonlinear potential: $\mathscr{V}=\frac{x^2}{10}-|u|^2$,
\end{itemize}
which give rise to solutions that propagate to the right side and undergoes dispersion (see Figure \ref{solinit}). This first subsection devotes to the comparison of the classical algorithm and the new algorithm. In section \ref{sec_vtx}, we mainly study the spectral properties of the classical algorithm and the preconditioned algorithms, i.e. the eigenvalues of the operators $I-\mathcal{L}_{h,n}$ and $P^{-1}(I-\mathcal{L}_{h,n})$. The last one compares the classical algorithm and the preconditioned algorithm for the general nonlinear potential.

\begin{figure}[!htbp] 
\centering
\includegraphics[width=0.7\textwidth]{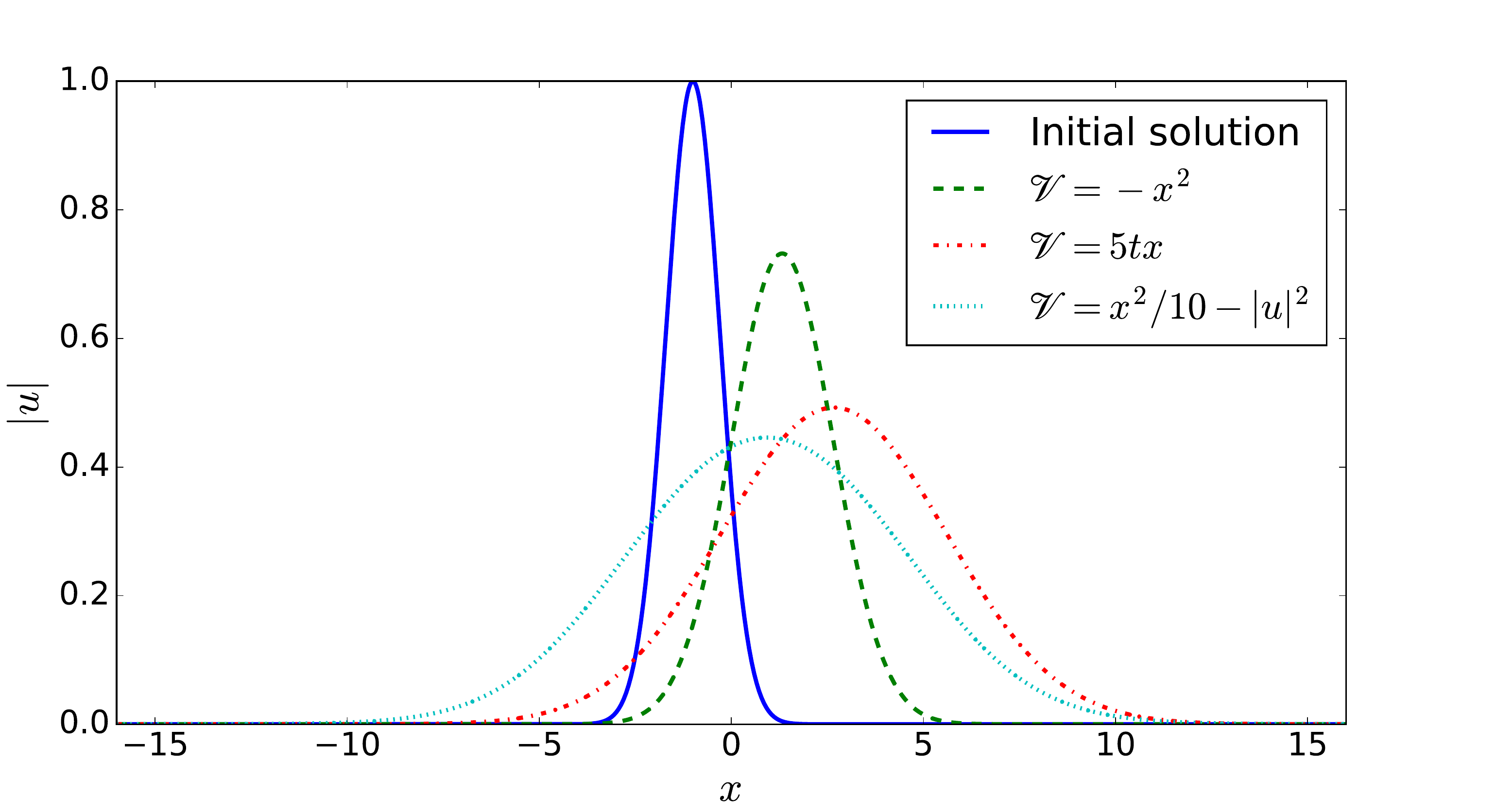}
\caption{Initial solution $|u_0|$ and numerical solutions of the Sch\"odinger equation with different potentials at final time $T=1.0$. The time step is $\Delta t=0.001$. The meshes are $\Delta x=1.0\times 10^{-5}$ for $\mathscr{V}=-x^2$ and $\Delta x=5.0\times 10^{-5}$ for the others.}
\label{solinit}
\end{figure}

All the numerical tests are implemented in a CPU cluster (16 cores/node, Intel Sandy Bridge E5-2670, 64GB memory/node). We fix always one core per MPI process, one MPI process per sub domain and 16 MPI processes per node. The communications are handled by OpenMPI 1.6.5 (GCC 4.8). The linear systems \eqref{distrobin} and \eqref{AvMNL} related to the finite element method are solved by the LU direct method using the MKL Pardiso library. 

\begin{rmk}
  Without less of generality, we consider the number of iterations of the first time step. In other words, we study the number of iterations of the evolution between $t_0 \rightarrow t_1$ to compute $v^1$ with the optimized Schwarz method.
\end{rmk}

\begin{rmk}
\label{rmk_zerorandom}
  Using the zero vector as the initial guess vector $g^0$ is normally more time efficient. While as mentioned in \cite{Gander2008history}, it could give wrong conclusions associated with the convergence. Thus, the zero guess vector is used when one wants to evaluate the computation time, while the random guess vector is used when study the number of iterations.
\end{rmk}

\begin{rmk}
  The theoretical optimal parameter $p$ in the Robin transmission condition is not at hand for us. We then seek for the best parameter numerically for each case if necessary.
\end{rmk}

\subsection{Time-independent linear potential}
\label{sec_vx}

In this part, we are interested in observing the efficiency of our new algorithms for the time-independent linear potential $\mathscr{V}=-x^2$. We denote by $T_{\mathrm{ref}}$ the computation time to solve the Schr{\"o}dinger equation numerically on a single processor on the entire domain and $T_{\mathrm{cls}}$ (resp. $T_{\mathrm{new}}$) the computation time of the classical (resp. new) algorithm for $N$ sub domains. We test the algorithms for $N=2, 4, 8, 16, 32, 64, 128, 256$ sub domains. 

The computation times are shown in Table \ref{Time_xx} where the fixed point method, two Krylov methods (GMRES and BiCGStab) and the LU direct method (sequential implementation and parallel implementation) are used on the interface problem (Step 2.2 in the algorithm \ref{Algo_new}). 
The time step and the mesh are $\Delta t=0.001$, $\Delta x = 10^{-5}$. Roughly speaking, in Table \ref{Time_xx} we have
\begin{align*}
  & T_{\mathrm{cls}} = T_{\mathrm{sub}} \times N_{\mathrm{iter}} \times N_T + ...,\\
  & T_{\mathrm{new}} = T_{\mathrm{sub}} \times 2 + (T_{\mathrm{sub}} \times 2 +  T_{Ld}) \times N_T + ...,
\end{align*}
where $T_{\mathrm{sub}}$ denotes the computation time for solving the equation on one sub domain for one time step, $T_{Ld}$ is the computation time for solving the interface problem, ``$...$''  represent the negligible part of computation time such as the construction of matrices for the finite element method. 
Firstly we can see that $T_{\mathrm{new}}$ with the parallel LU method takes much more times than others if $N$ is large. 
As has been explained in section \ref{sec_new} (Choice 2), since the LU method is inherently sequential and the size of \eqref{newalgoILd} is quite small, the communication is costly in our case. 
From now on, we abandon our idea using a parallel LU direct solver on the interface problem. 
For other choices, if the number of sub domains $N$ is not so large, then $T_{\mathrm{sub}} \gg T_{Ld}$ and the minimum of $N_{\mathrm{iter}}$ is 3 in all our tests (see Table \ref{Niter_xx}). If the number of sub domains $N$ is large, then $T_{Ld} \sim T_{\mathrm{sub}}$ and $N_{\mathrm{iter}} \geqslant 3$. It is for this reason that the new algorithm takes less computation time and we do not observe big difference when different solvers used on the interface problem in the new algorithm.
However, we emphasize that the new algorithm with the sequential LU method on the interface problem is the best choice. It allows to use any $p\in\mathbb{R}^+$. In other words, the algorithm is independent of the transmission condition.
Note that it is possible to use other more complex transmission condition as presented in \cite{Halpern2010_sch, Antoine2014, XF20151d}. Normally, this leads to better convergence properties of the classical algorithms. However, the new algorithm with the direct method on the interface problem is a direct algorithm.  

\begin{table}[!htbp] 
  \caption{Computation time (seconds) of the classical algorithm and the new algorithm with the fixed point method, two Krylov methods (GMRES and BiCGStab) and the LU direct method (sequential implementation and parallel implementation) where $\mathscr{V}=-x^2$, $\Delta t=0.001$, $\Delta x = 10^{-5}$ and $p=45$.}
  \centering
  \renewcommand\tabcolsep{3.0pt}
  \renewcommand{\arraystretch}{1.1}
  \begin{tabular}{|c|c|c|c|c|c|c|c|c|c|c|}
    \hline
    & $N$ & 2 & 4 & 8 & 16 & 32 & 64 & 128 & 256  \\
    \cline{2-11}  
    & $T_{\mathrm{ref}}$ &  \multicolumn{8}{c|}{594.0} \\
    \hline
    \multirow{2}{*}{Fixed Point}
    & $T_{\mathrm{cls}}$ & 11102.0 & 5909.4 & 2948.7 & 1596.5 & 788.9 & 412.7 & 199.3 & 105.4 \\
    & $T_{\mathrm{new}}$ & 730.0 & 374.7 & 185.1 & 101.5 & 49.0 & 26.1 & 13.4 & 7.9  \\
    \hline
    \multirow{2}{*}{GMRES}
    & $T_{\mathrm{cls}}$ & 1461.5 & 934.5 & 611.4 & 423.3 & 220.0 & 116.4 & 55.2 & 28.1 \\
    & $T_{\mathrm{new}}$ & 728.8 & 375.5 & 185.6 & 101.9 & 48.9 & 25.7 & 12.7 & 7.0 \\
    \hline
    \multirow{2}{*}{BiCGStab}
    & $T_{\mathrm{cls}}$ & 1807.2 & 1161.5 & 795.9 & 489.4 & 235.3 & 115.8 & 50.7 & 31.2 \\
    & $T_{\mathrm{new}}$ & 727.7 & 375.4 & 185.3 & 101.4 & 48.8 & 25.5 & 12.7 & 6.9 \\
    \hline
    \multirow{1}{*}{LU-S} 
    & $T_{\mathrm{new}}$ & 663.3 & 363.6 & 183.9 & 101.0 & 48.4 & 25.0 & 12.1 & 6.3 \\
    \hline
    \multirow{1}{*}{LU-P} 
    & $T_{\mathrm{new}}$ & 680.8 & 356.8 & 189.3 & 102.9 & 51.3 & 33.6 & 31.4 & 65.2 \\
    \hline
  \end{tabular} 
  \begin{tablenotes}
  \small
  \item \hspace{-5mm} LU-S: LU direct method (sequential implementation) in MKL Pardiso library.
  \item \hspace{-5mm} LU-P: LU direct method (parallel implementation, MPI) in MUMPS library \cite{mumps}.
  \end{tablenotes}
  \label{Time_xx}
\end{table}

\begin{table}[!htbp] 
  \caption{Number of iterations with the iterative methods used on the interface problem for some different $p$ where $\mathscr{V}=-x^2$, $\Delta t=0.001$ and $\Delta x = 10^{-5}$.}
  \label{Niter_xx}
  \centering 
  \renewcommand\tabcolsep{4.5pt}
  \begin{tabular}{|c|c|c|c|c|c|c|c|c|c|c|}
    \hline
    & \multicolumn{3}{c|}{$N=2$} & \multicolumn{3}{c|}{$N=256$}  \\
    \hline
    $p$ & Fixed point & GMRES & BiCGStab & Fixed point & GMRES & BiCGStab  \\
    \hline
    5 & 159 & 3 & 3 & 185 & 15 & 9 \\
10 & 83 & 3 & 3 & 96 & 15 & 9 \\
15 & 58 & 3 & 3 & 67 & 14 & 8 \\
20 & 45 & 3 & 3 & 52 & 14 & 8 \\
25 & 39 & 3 & 3 & 44 & 13 & 8 \\
30 & 35 & 3 & 3 & 40 & 13 & 8 \\
35 & 33 & 3 & 3 & 37 & 13 & 8 \\
40 & 32 & 3 & 3 & 36 & 13 & 7 \\
45 & 31 & 3 & 3 & 35 & 13 & 7 \\
50 & 32 & 3 & 3 & 35 & 13 & 7 \\
    \hline
  \end{tabular}
\end{table}

In conclusion, the new algorithm with the sequential LU direct method is robust, scalable and independent of the transmission condition. In addition, it takes less computation time than the classical algorithm.

\subsection{Time-dependent linear potential}
\label{sec_vtx}

In this subsection, we consider the preconditioned algorithm and the classical algorithm for the time-dependent potential $\mathscr{V}=5tx$. 
According to our numerical investigations in the previous subsection, the preconditioner \eqref{Pxg} is implemented using the choice 3 in section \ref{sec_new}, which is robust and time efficient. 

Firstly, we are interested in the efficiency of the preconditioner. The explicit construction of the interface problem allows us to study the spectral properties of $I-\mathcal{L}_{h,n}$ and $P^{-1}(I-\mathcal{L}_{h,n})$. Figure \ref{eigvalmesh2} and Figure \ref{eigvalmesh256} present their eigenvalues for some different meshes where $N=2$ and $N=256$ respectively. Without less of generality, we only consider $n=1$. It can be seen that the eigenvalues of the preconditioned linear system are really close to 1+0i, which illustrates that $P^{-1}$ is a good preconditioner.
\begin{figure}[!htbp]
  \centering
  \begin{subfigure}{0.48\textwidth}
    \centering
    \includegraphics[width=\textwidth]{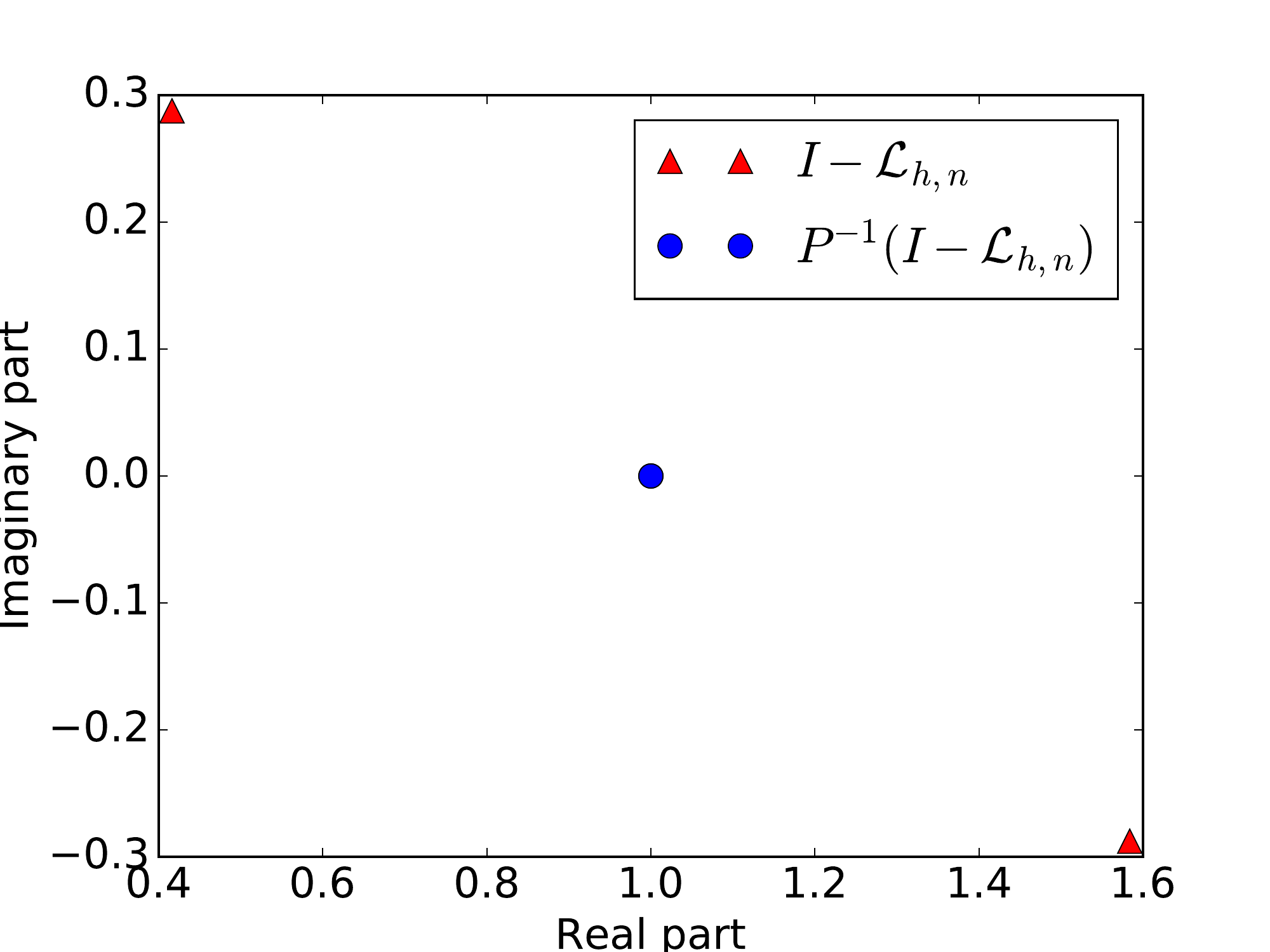}
    \caption{$\Delta t=0.01, \Delta x=0.01$}
  \end{subfigure}
  \begin{subfigure}{0.48\textwidth}
    \centering
    \includegraphics[width=\textwidth]{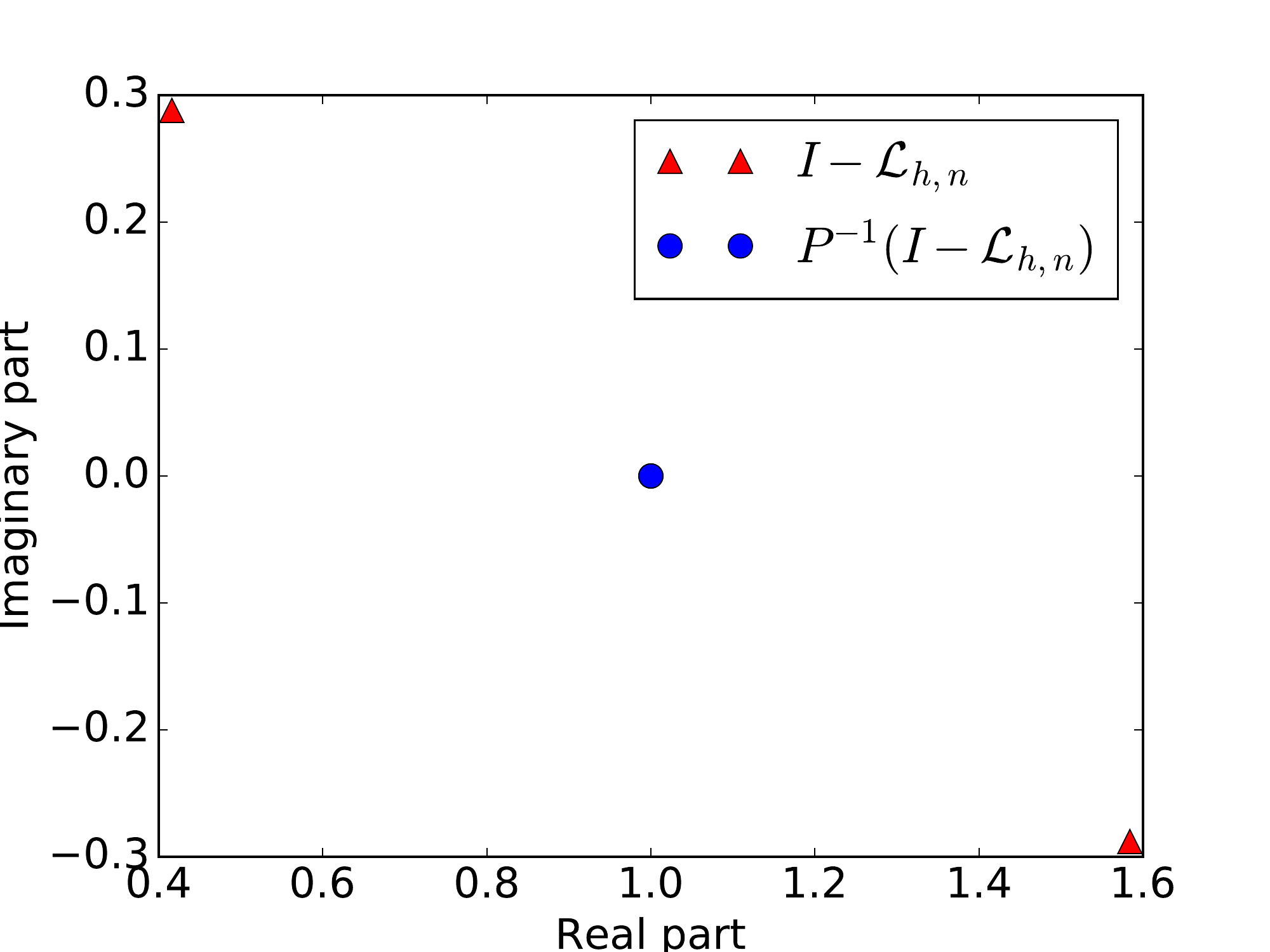}
    \caption{$\Delta t=0.01, \Delta x=5\times 10^{-5}$}
  \end{subfigure}
  \begin{subfigure}{0.48\textwidth}
    \centering
    \includegraphics[width=\textwidth]{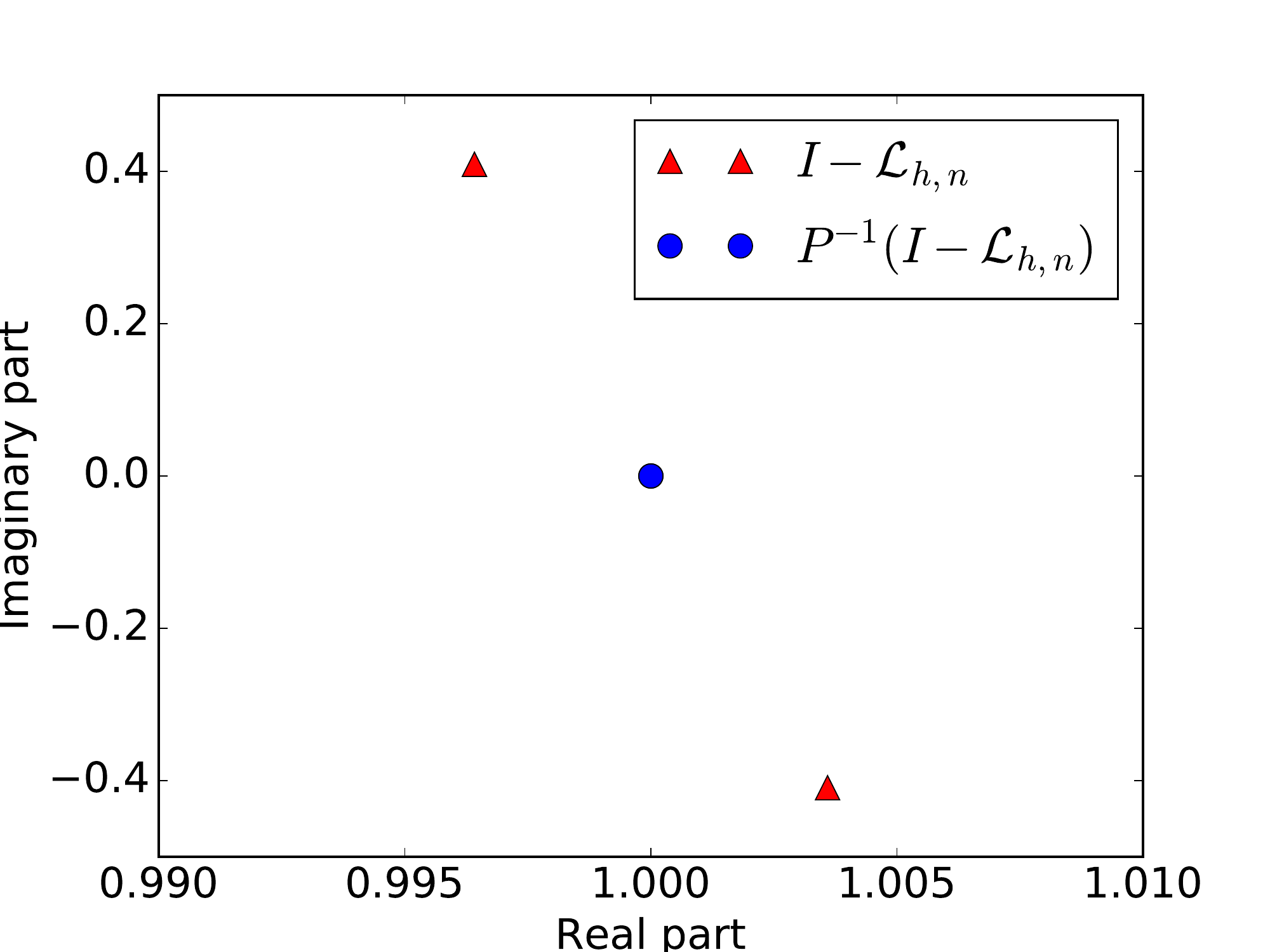}
    \caption{$\Delta t=0.001, \Delta x=0.01$}
  \end{subfigure}
  \begin{subfigure}{0.48\textwidth}
    \centering
    \includegraphics[width=\textwidth]{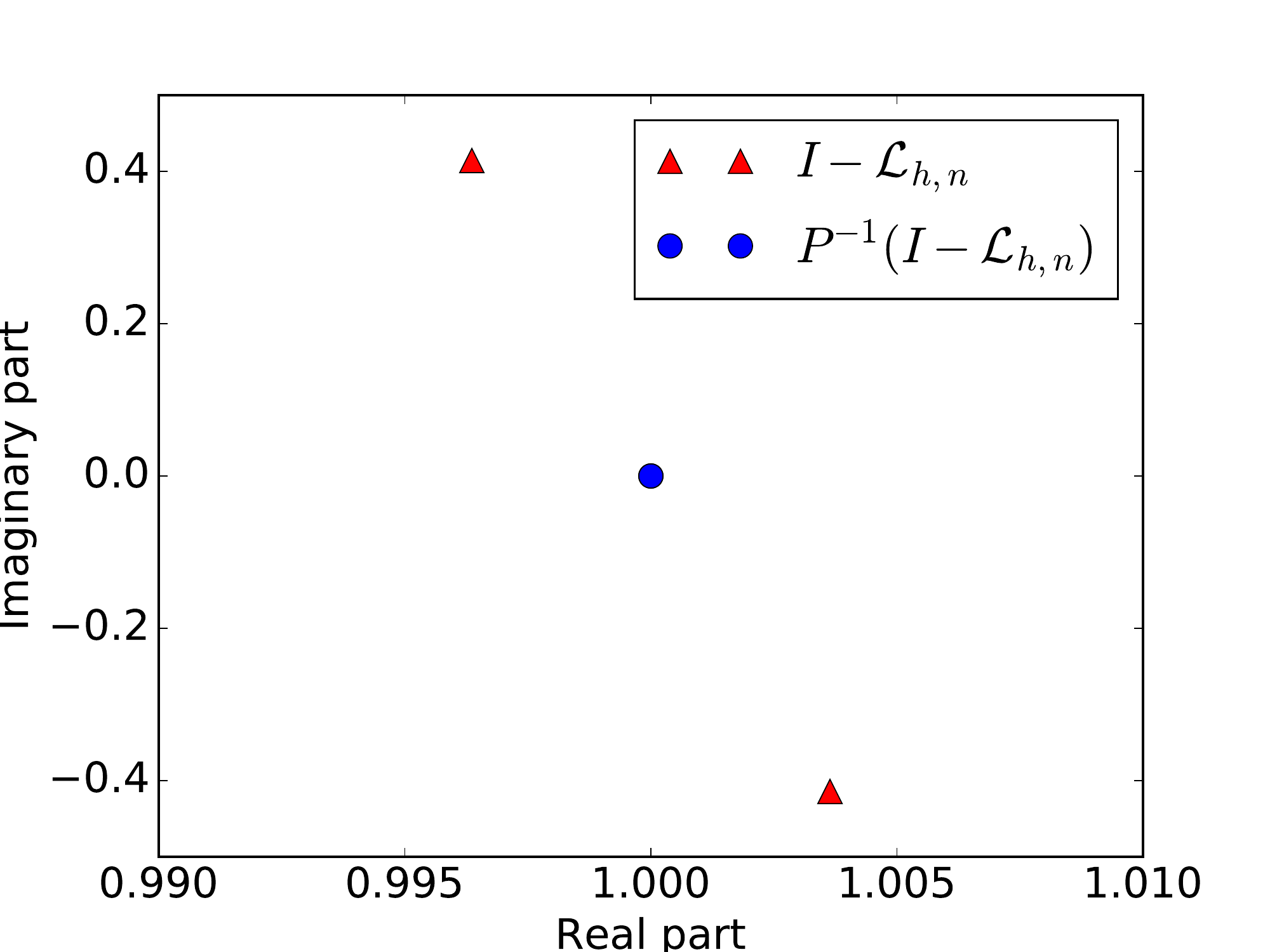}
    \caption{$\Delta t=0.001, \Delta x=5\times 10^{-5}$}
  \end{subfigure}
  \caption{Eigenvalues of $I-\mathcal{L}_{h,n}$ and $P^{-1}(I-\mathcal{L}_{h,n})$ for $N=2$ where $n=1$ and $p=45$.}
  \label{eigvalmesh2}
\end{figure}
\begin{figure}[!htbp]
  \centering
  \begin{subfigure}{0.48\textwidth}
    \centering
    \includegraphics[width=\textwidth]{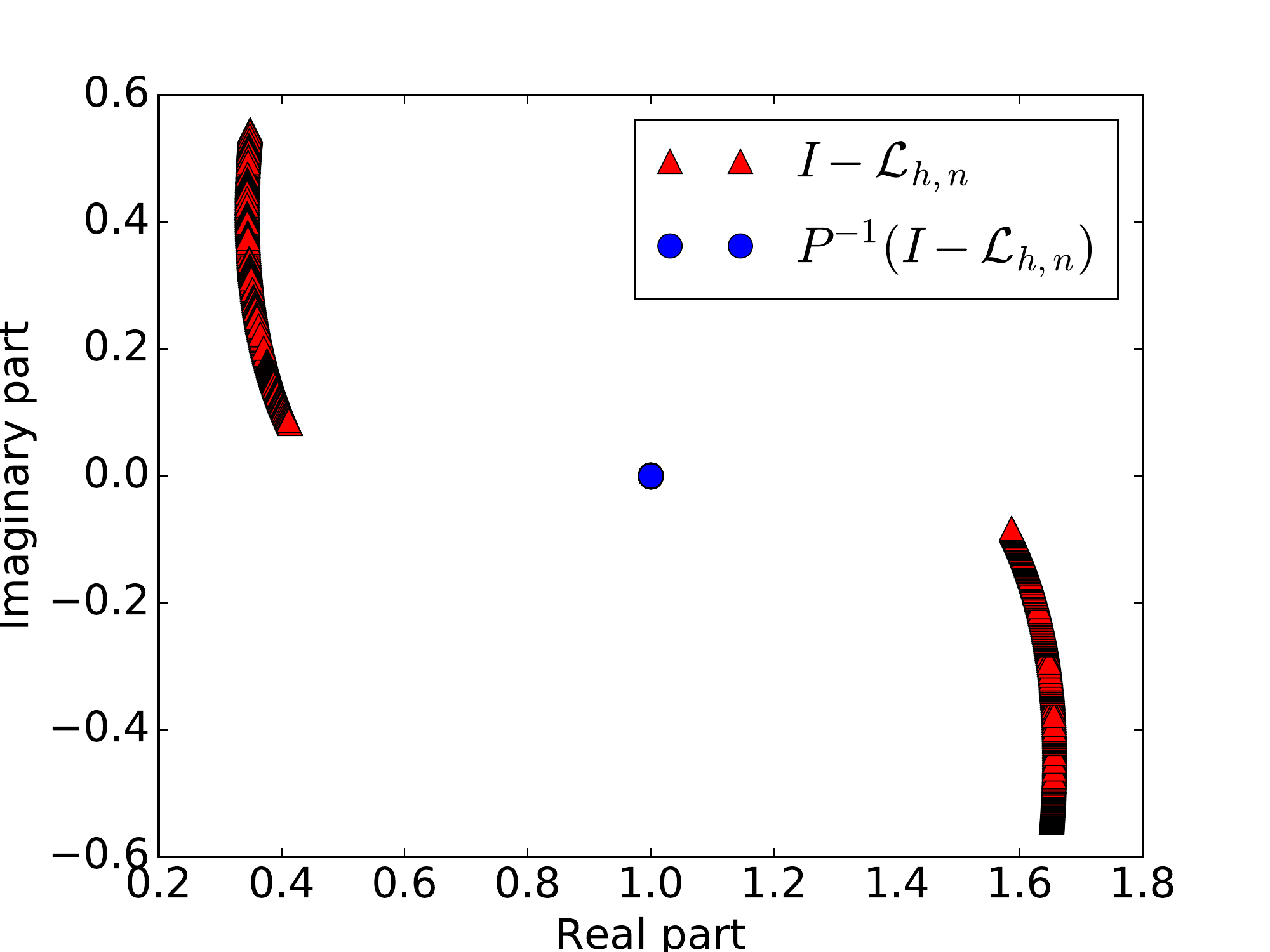}
    \caption{$\Delta t=0.01, \Delta x=5\times 10^{-5}$}
  \end{subfigure}
    \begin{subfigure}{0.48\textwidth}
    \centering
    \includegraphics[width=\textwidth]{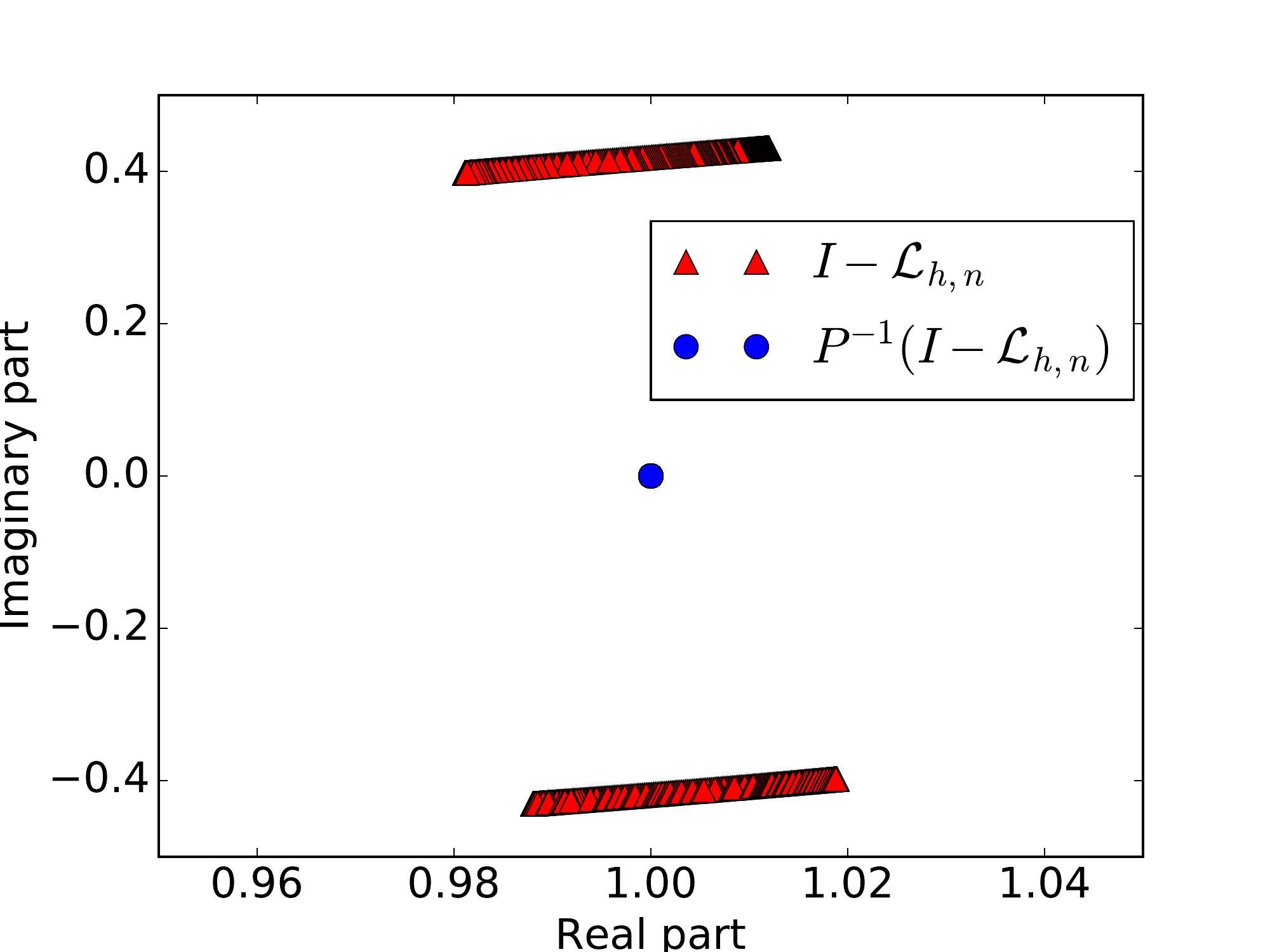}
    \caption{$\Delta t=0.001, \Delta x=5\times 10^{-5}$}
  \end{subfigure}
  \caption{Eigenvalues of $I-\mathcal{L}_{h,n}$ and $P^{-1}(I-\mathcal{L}_{h,n})$ for $N=256$ where $n=1$ and $p=45$.}
  \label{eigvalmesh256}
\end{figure}

Next, we show in Figure \ref{eigvalp2} and \ref{eigvalp256} the eigenvalues of $I-\mathcal{L}_{h,n}$ and $P^{-1}(I-\mathcal{L}_{h,n})$ for some different $p$ and for $N=2,256$ respectively. It can be seen that the parameter $p$ has almost no influence on $P^{-1}(I-\mathcal{L}_{h,n})$. 
As for $I-\mathcal{L}_{h,n}$, the distance between the eigenvalues and the value 1+0i depends on $p$. Thus, if the fixed point method is used on the interface problem, then it is essential to choose the optimal parameter $p$ in the classical algorithm. However, as shown in Table \ref{nbitervtx2} and \ref{nbitervtx256} (number of iterations of the classical algorithm and the preconditioned algorithm for $N=2$ and $N=256$), the use of the Krylov methods (GMRES and BiCGStab) can also reduce the dependency on the parameter $p$ in the transmission condition.
\begin{figure}[!htbp]
\centering
  \includegraphics[width=0.48\textwidth]{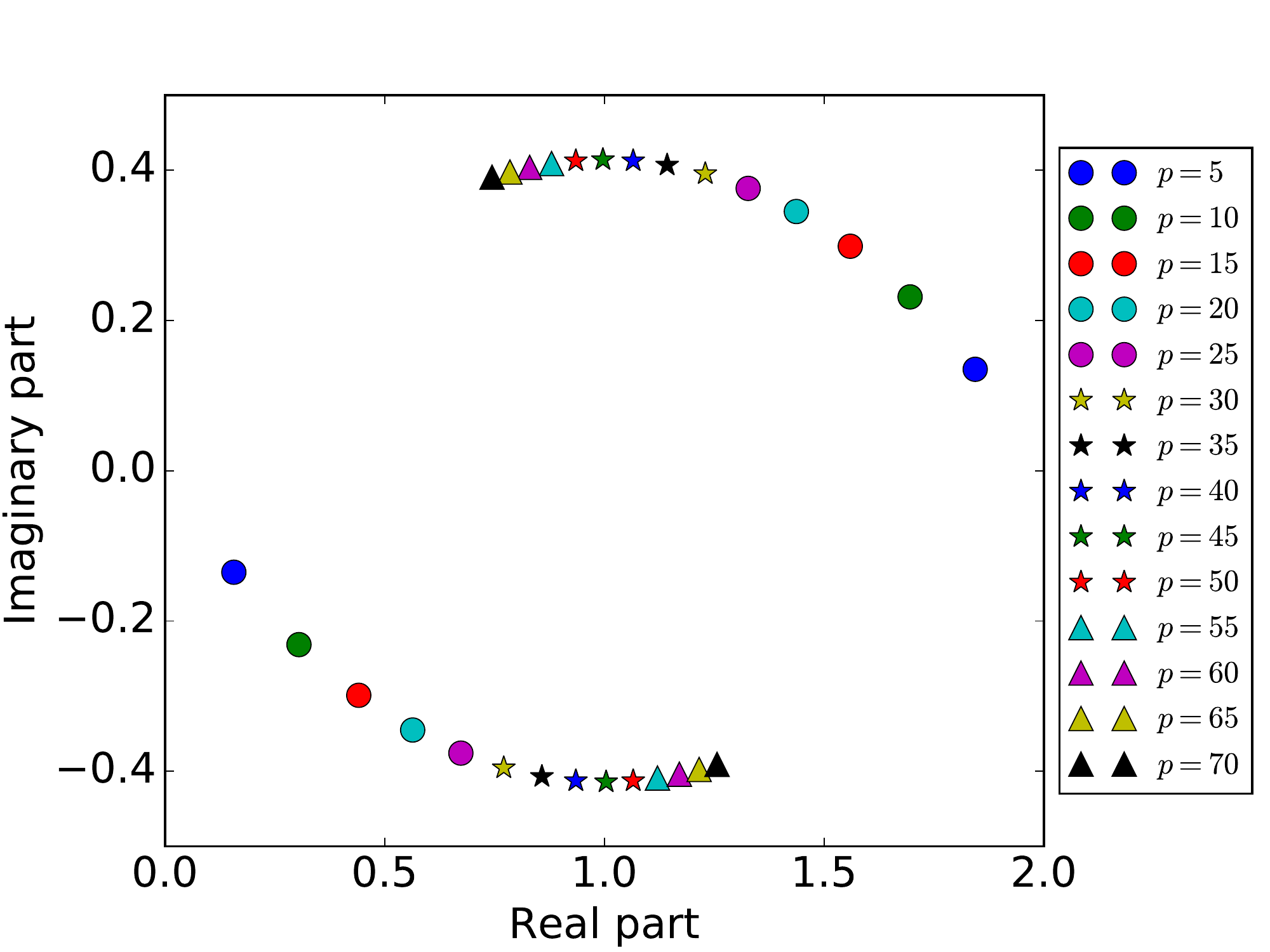}
  \includegraphics[width=0.48\textwidth]{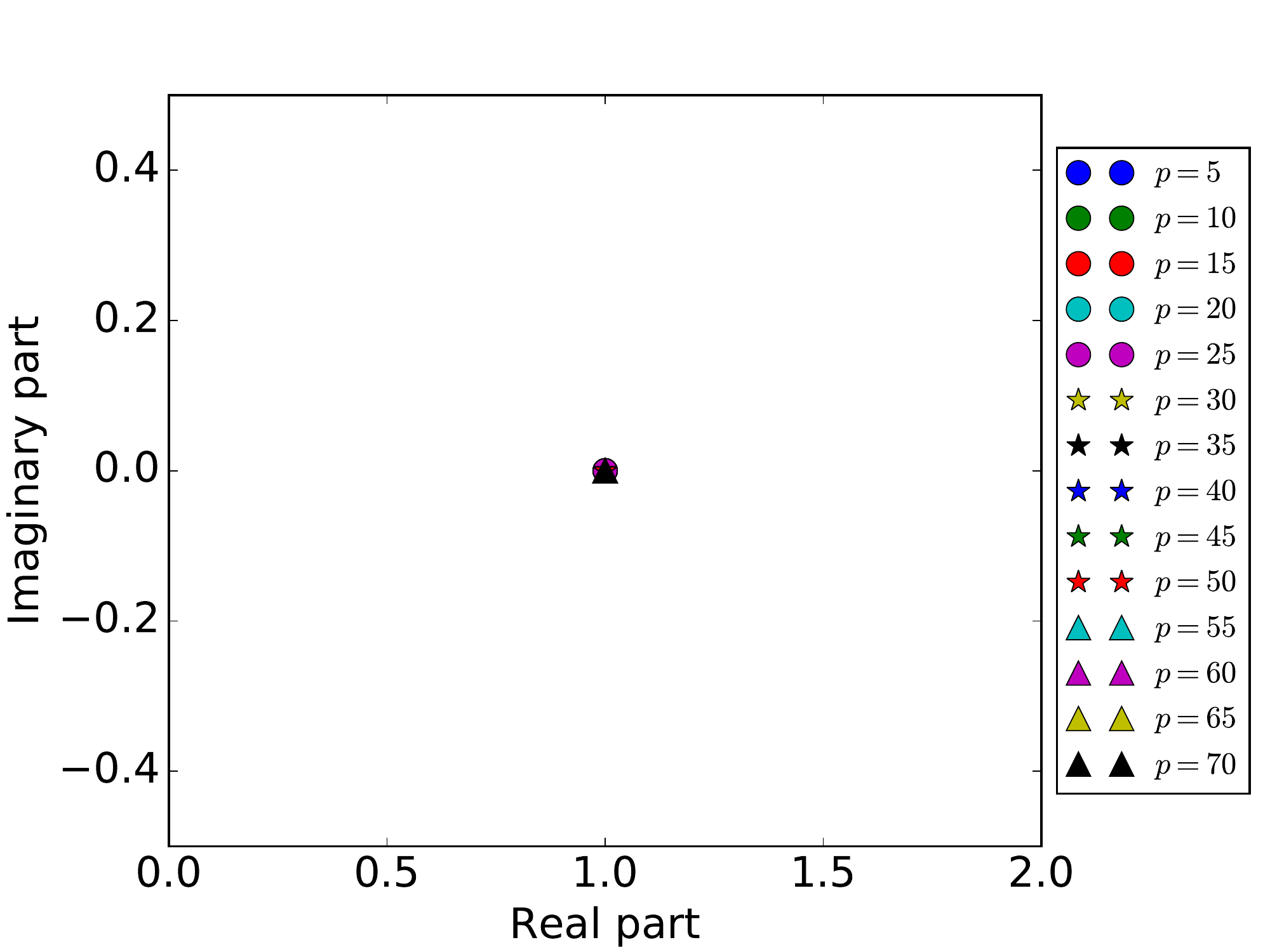}
  \caption{Eigenvalues of $I-\mathcal{L}_{h,n}$ (left) and $P^{-1}(I-\mathcal{L}_{h,n})$ (right) for different $p$ where $N=2$ and $\Delta t=0.001, \Delta x=5\times 10^{-5}$.}
  \label{eigvalp2}
\end{figure}
\begin{figure}[!htbp]
\centering
  \includegraphics[width=0.48\textwidth]{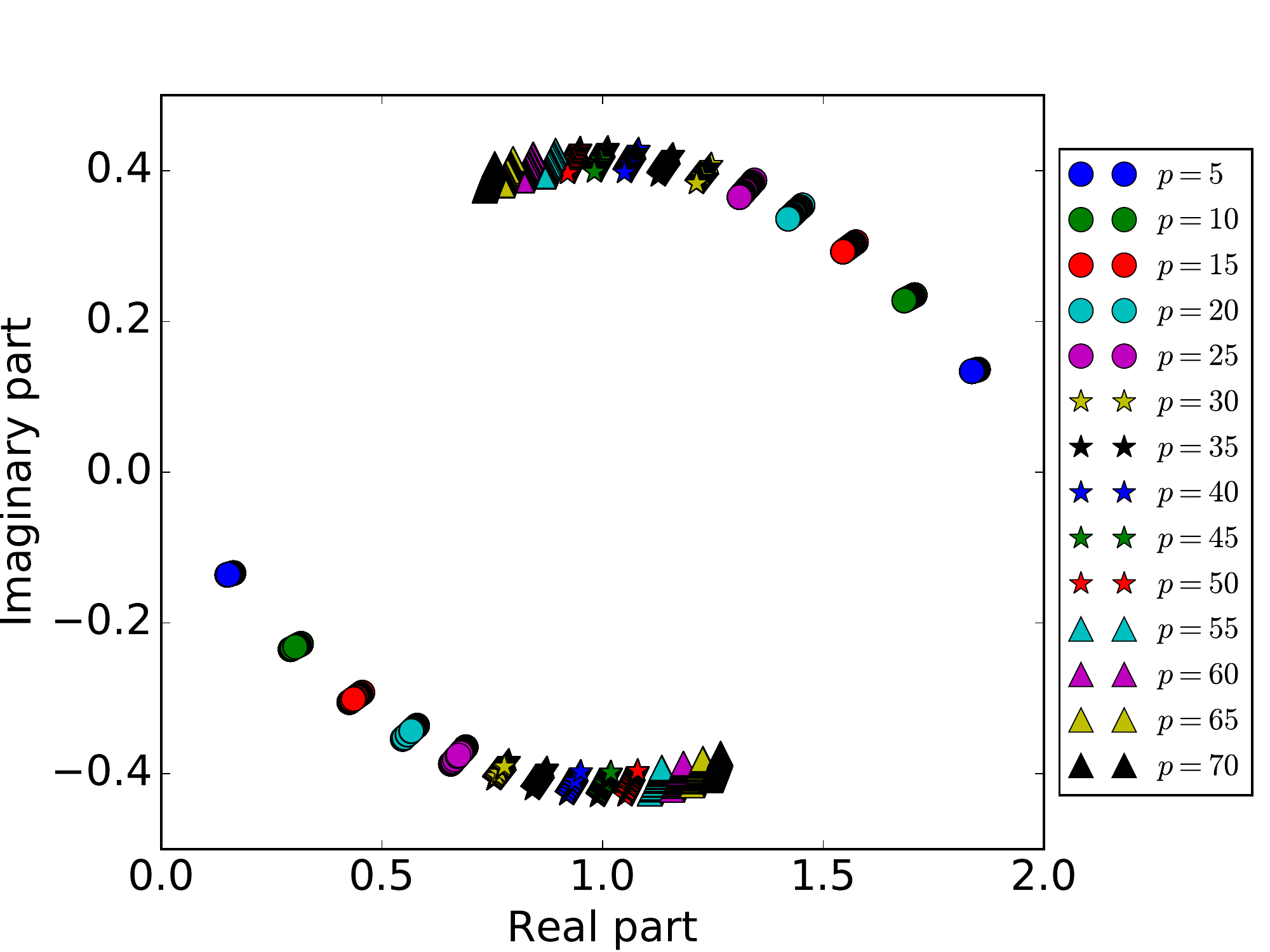}
  \includegraphics[width=0.48\textwidth]{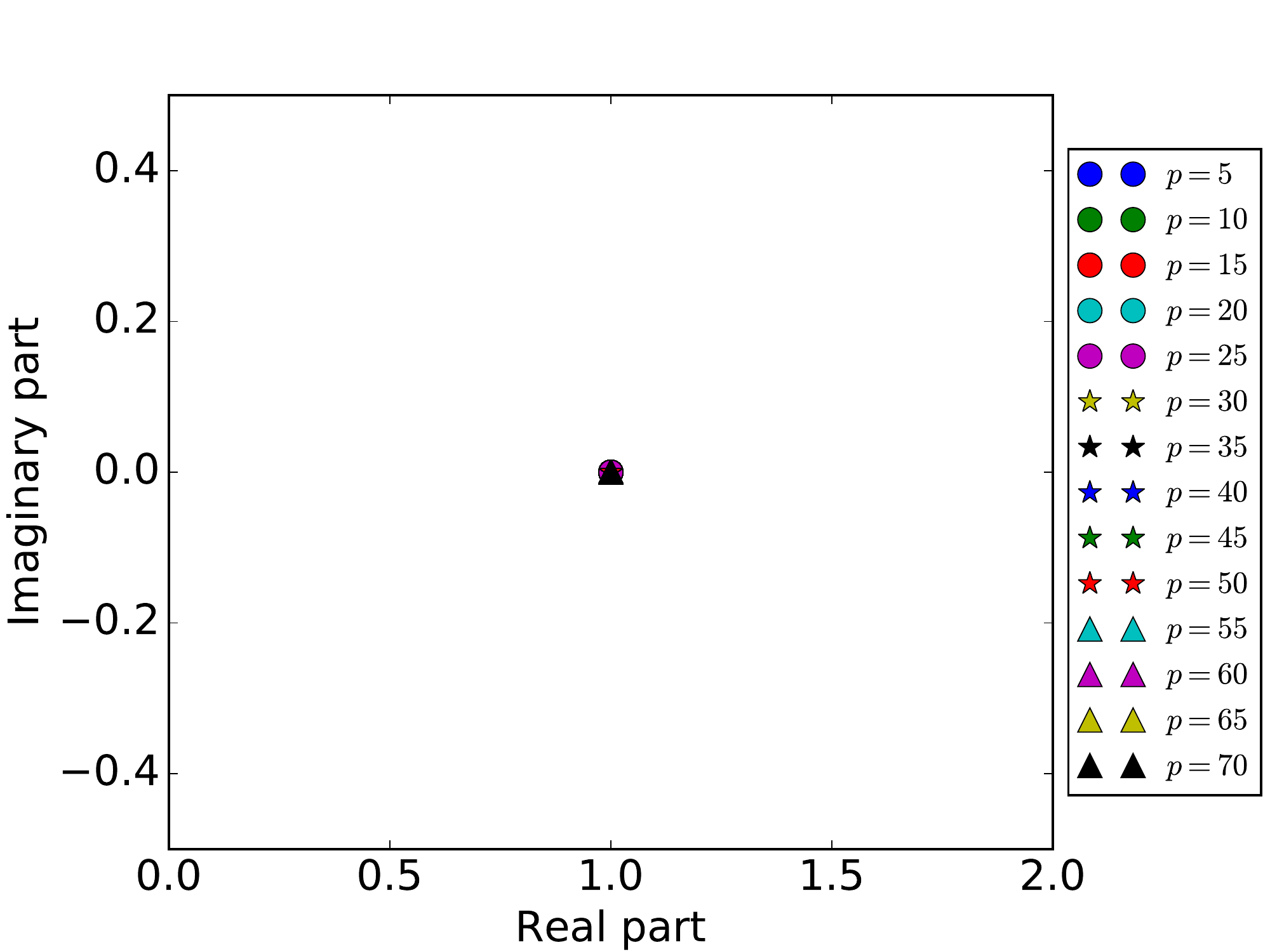}
  \caption{Eigenvalues of $I-\mathcal{L}_{h,n}$ (left) and $P^{-1}(I-\mathcal{L}_{h,n})$ (right) for different $p$ where $N=256$ and $\Delta t=0.001, \Delta x=5\times 10^{-5}$.}
  \label{eigvalp256}
\end{figure}
\begin{table}[!htbp]
  \caption{Number of iterations of the classical algorithm and the preconditioned algorithm (+PC) for $N=2$ where $\mathscr{V}=5tx$, $\Delta t=0.001$ and $\Delta x = 5 \times 10^{-5}$.}
  \centering
  \begin{tabular}{|c|c|c|c|c|c|c|c|c|c|c|c|}
    \hline
    $p$ & 5 & 10 & 15 & 20 & 25 & 30 & 35 & 40 & 45 & 50  \\
    \hline
    Fixed point & 158 & 82 & 58 & 45 & 39 & 35 & 33 & 32 & 31 & 32 \\
    \hline
    Fixed point + PC & 3 & 3 & 3 & 3 & 3 & 3 & 3 & 3 & 3 & 3\\
    \hline
    GMRES &  3 & 3 & 3 & 3 & 3 & 3 & 3 & 3 & 3 & 3 \\
    \hline
    GMRES+PC & 3 & 3 & 3 & 3 & 3 & 3 & 3 & 3 & 3 & 3 \\
    \hline
    BiCGStab & 3 & 3 & 3 & 3 & 3 & 3 & 3 & 3 & 3 & 3 \\
    \hline
    BiCGStab + PC & 2 & 2 & 2 & 2 & 2 & 2 & 2 & 2 & 2 & 2 \\
    \hline
  \end{tabular}
  \label{nbitervtx2}
\end{table}
\begin{table}[!htbp]
  \caption{Number of iterations of the classical algorithm and the preconditioned algorithm (+PC) for $N=256$ where $\mathscr{V}=5tx$, $\Delta t=0.001$ and $\Delta x = 5 \times 10^{-5}$.}
  \centering
  \begin{tabular}{|c|c|c|c|c|c|c|c|c|c|c|c|}
    \hline
    $p$ & 5 & 10 & 15 & 20 & 25 & 30 & 35 & 40 & 45 & 50  \\
    \hline
    Fixed point & 184 & 95 & 66 & 52 & 44 & 40 & 37 & 36 & 35 & 35 \\
    \hline
    Fixed point + PC & 4 & 4 & 4 & 4 & 4 & 3 & 4 & 4 & 4 & 4 \\
    \hline
    GMRES & 14 & 12 & 13 & 12 & 12 & 12 & 12 & 12 & 12 & 12 \\
    \hline
    GMRES+PC & 4 & 4 & 3 & 4 & 4 & 4 & 4 & 4 & 4 & 4 \\
    \hline
    BiCGStab & 8 & 8 & 7 & 7 & 7 & 7 & 7 & 7 & 7 & 7 \\
    \hline
    BiCGStab + PC & 3 & 3 & 3 & 3 & 3 & 3 & 3 & 3 & 3 & 2 \\
    \hline
  \end{tabular}
  \label{nbitervtx256}
\end{table}

Finally, Table \ref{time_tx} presents the computation times of the classical algorithm and the preconditioned algorithm where the fixed point method, the GMRES method and the BiCGStab method are used on the interface problem. The computation time of the preconditioned algorithm $T_{\mathrm{pc}}$ consists of three major parts: the construction of the preconditioner (denoted by $T_{1}$), the application of $\mathcal{R}_{h,n}$ to vectors (denoted by $T_2$), and the application of preconditioner (denoted by $T_{3}$). Then we have
\begin{equation}
  \label{Tpc_123}
  T_{\mathrm{pc}} \approx  T_{1} + (T_2+T_3) \times N_{\mathrm{iter}} \times N_T,
\end{equation}
where $T_{2}\gg T_3$, $T_1 \approx 2T_2$. As can be seen from Table \ref{nbitervtx2} and Table \ref{nbitervtx256}, if the fixed point method is used on the interface problem, the preconditioner allows to reduce significantly the number of iterations required for convergence. Thus, $T_{\mathrm{pc}} < T_{\mathrm{cls}}$.
\begin{table}[!htbp] 
  \centering
  \caption{Computation time (seconds) of the classical algorithm and the preconditioned algorithm (+PC) where $\mathscr{V}=5tx$, $\Delta t = 0.001$ and $\Delta x = 5 \times 10^{-5}$.}
  \renewcommand\tabcolsep{5.0pt}
  \begin{tabular}{|c|c|c|c|c|c|c|c|c|c|}
    \hline
    $N$ & 2 & 4 & 8 & 16 & 32 & 64 & 128 & 256 \\
    \hline
    $T_{\mathrm{ref}}$ & \multicolumn{8}{c|}{2100.4}\\
    \hline
    Fixed point & 3438.5 & 1732.8 & 840.0 & 413.9 & 193.7 & 80.1 & 45.9 & 24.4\\
    \hline
    Fixed point+PC & 1400.3 & 692.6 & 332.1 & 159.7 & 71.6 & 28.8 & 16.5 & 9.9 \\
    \hline
    GMRES & 1250.0 & 712.5 & 376.8 & 188.0 & 85.3 & 34.2 & 19.3 & 10.4 \\
    \hline
    GMRES+PC & 1246.1 & 664.9 & 331.9 & 157.7 & 70.6 & 28.4 & 16.2 & 10.0 \\
    \hline
    BiCGStab & 1235.1 & 710.8 & 345.8 & 179.2 & 81.1 & 36.9 & 17.7 & 10.7 \\
    \hline
    BiCGStab+PC & 1326.7 & 595.0 & 287.7 & 147.9 & 68.1 & 31.1 & 15.2 & 10.8\\
    \hline
  \end{tabular}
  \label{time_tx}
\end{table}

\begin{rmk}
\label{rmk_lin}
Let us read Table \ref{nbitervtx256} and Table \ref{time_tx} together for $N=256$ and $p=45 $ with the fixed point method used on the interface problem. The reduction of computation time is not proportional to that of number of iterations.
In fact, the number of iterations of the preconditioned algorithm is about 10 times less than that of the classical one. However, its computation time is only reduced by about 2.5 times. There are three reasons:
\begin{itemize}[leftmargin=*]
\item For each time step $n$, the linear systems \eqref{distrobin} or \eqref{AvMNL} are solved multiple times. However, it is enough to do one time the LU factorization of the matrix $\mathbb{A}_{j,n} + ip \cdot \mathbb{M}^{\Gamma_j}$, of which the computation times are same for the classical algorithm and the preconditioned algorithm.
\item As explained in remark \ref{rmk_zerorandom}, the initial vector used in Table \ref{nbitervtx256} is a random vector, while it is the zero vector in Table \ref{time_tx}. 
\item The implementation of the preconditioner takes a little computation time.
\end{itemize}
\end{rmk}

Before turning to consider the Schr\"odinger equation with a general nonlinear potential, we can conclude that $P^{-1}$ is a very efficient preconditioner. In addition, the use of Krylov methods on the interface problem allows to reduce both of the number of iterations and the computation time in the framework of the classical algorithm. 

\subsection{Nonlinear potential}
\label{sec_vnl}

The objective of the last subsection is to investigate the performance of the preconditioned algorithm for the nonlinear potential $\mathscr{V}=x^2/10-|u|^2$. 
Unlike the linear case, the interface problem \eqref{algo_d} is nonlinear here. 
Thus, it is not possible to use Krylov methods to accelerate the convergence. Table \ref{time_vnl} presents the computation time of the classical algorithm ($T_{\mathrm{cls}}$) and the preconditioned algorithm ($T_{\mathrm{pc}}$) for $N=2,4,8,16,32,64,128,256$ where the time step is $\Delta t = 0.001$, the mesh is $\Delta x = 5 \times 10^{-5}$ and $p=45$. 
We can see that both of the algorithms are scalable and the preconditioner can significantly reduce the computation time.
\begin{table}[!htbp] 
  \centering
  \caption{Computation time (seconds) of the classical algorithm and the preconditioned algorithm where $\mathscr{V}= x^2/10-|u|^2$, $\Delta t = 0.001$, $\Delta x = 5 \times 10^{-5}$ and $p=45$.}
  \label{time_vnl}
  \begin{tabular}{|c|c|c|c|c|c|c|c|c|c|}
    \hline
    $N$ & 2 & 4 & 8 & 16 & 32 & 64 & 128 & 256  \\
    \hline
    $T_{\mathrm{ref}}$ & \multicolumn{9}{c|}{2423.9}\\
    \hline
    $T_{\mathrm{cls}}$ & 11152.7 & 6147.9 & 3074.4 & 1615.2 & 767.0 & 416.6 & 202.8 & 98.0 \\
    \hline
    $T_{\mathrm{pc}}$ & 2593.4 & 1362.3 & 648.3 & 318.2 & 153.7 & 81.0 & 39.1 & 18.7   \\
    \hline
  \end{tabular}
\end{table}

We show in Table \ref{nbiterpnl} the number of iterations of the classical algorithm ($N_{\mathrm{cls}}$) and the preconditioned  algorithm ($N_{\mathrm{pc}}$) with different parameter $p$ in the transmission condition for $N=2$ and $N=256$. 
As the linear case, the preconditioned algorithm is almost independent of the parameter $p$. 
In addition, the preconditioned algorithm needs much less iterations to converge.
Note that it is for the same reasons as the linear case that the reduction of computation time is not proportional to that of number of iterations (see remark \ref{rmk_lin}). 

\begin{table}[!htbp] 
  \caption{Number of iterations of the classical algorithm and the preconditioned algorithm for some different $p$ where $\mathscr{V}= x^2/10-|u|^2$, $\Delta t = 0.001$ and $\Delta x = 5 \times 10^{-5}$.}
  \label{nbiterpnl}
  \centering
  \begin{tabular}{|c|c|c|c|c|c|c|c|c|c|c|}
    \hline
    & \multicolumn{2}{c|}{$N=2$} & \multicolumn{2}{c|}{$N=256$}  \\
    \hline
    $p$ & $N_{\mathrm{cls}}$ & $N_{\mathrm{pc}}$ & $N_{\mathrm{cls}}$ & $N_{\mathrm{pc}}$\\
    \hline
    5 & 147 & 3 & 170 & 3 \\
    10 & 79 & 3 & 90 & 3 \\
    15 & 55 & 3 & 63 & 4 \\
    20 & 44 & 3 & 50 & 4 \\
    25 & 38 & 3 & 43 & 4 \\
    30 & 34 & 3 & 39 & 4 \\
    35 & 32 & 3 & 36 & 4 \\
    40 & 31 & 3 & 35 & 4 \\
    45 & 30 & 3 & 34 & 4 \\
    50 & 31 & 3 & 34 & 4 \\
    \hline
  \end{tabular}
\end{table}

In conclusion, both of the classical algorithm and the preconditioned algorithm are robust. 
The preconditioned algorithm takes less number of iterations and less computation times than the classical one. In addition, the preconditioned algorithm is not sensible to the transmission condition.

\section{Conclusion}
\label{sec_conclusion}

We proposed in this paper a new optimized Schwarz algorithm for the one dimensional Schr\"odinger equation with time-independent linear potential and a preconditioned algorithm for the Schr\"odinger equation with general potential. These algorithms are robust and scalable. They could reduce significantly both the number of iterations and the computation time. In addition, by using the sequential LU direct method on the interface problem in the context of the new algorithm, the convergence is independent of the transmission condition. 

\section*{Acknowledgments} 

This work was granted access to the HPC and visualization resources of ``Centre de Calcul Interactif" hosted by University Nice Sophia Antipolis. 


\bibliography{Bib2}

\end{document}